\documentclass{amsart}
\usepackage{amsmath,amsthm,amsfonts,amssymb,graphicx,xcolor}
\usepackage[margin=1in]{geometry}
\usepackage{dsfont}
\usepackage{hyperref}

\newcommand{\one}{\mathds 1}
\DeclareMathOperator{\supp}{supp}
\DeclareMathOperator{\dist}{dist}

\newtheorem{theorem}{Theorem}[section]
\newtheorem{corollary}{Corollary}[theorem]
\newtheorem{lemma}[theorem]{Lemma}
\newtheorem{proposition}[theorem]{Proposition}

\theoremstyle{definition}
\newtheorem{definition}{Definition}[section]
\theoremstyle{remark}
\newtheorem{remark}{Remark}[section]
\numberwithin{equation}{section}

\newcommand{\abs}[1]{\left\vert#1\right\vert}
\newcommand{\nm}[1]{\left\Vert#1\right\Vert}
\newcommand{\br}[1]{\left\langle #1 \right\rangle}
\newcommand{\tnm}[1]{\left\Vert#1\right\Vert_{L^2}}
\newcommand{\lnm}[1]{\left\Vert#1\right\Vert_{L^{\infty}}}

\newcommand{\pnnm}[3]{\left\Vert#1\right\Vert_{L^{#2}_{#3}}}

\newcommand{\ud}{\mathrm{d}}
\newcommand{\p}{\partial}
\newcommand{\ls}{\lesssim}
\newcommand{\gs}{\gtrsim}
\newcommand{\rt}{\rightarrow}
\newcommand{\no}{\nonumber}
\newcommand{\ep}{\varepsilon}
\newcommand{\R}{\mathbb{R}}
\renewcommand{\S}{\mathbb{S}}
\newcommand{\vs}{v_{\ast}}
\newcommand{\li}{\mathcal{L}}
\newcommand{\fb}{f_b}
\newcommand{\fr}{f_r}
\newcommand{\Qff}{Q(f,f)}
\newcommand{\Hk}{H^{s,\gamma}_{kin}}
\newcommand{\trp}{\left(\partial_t + v \cdot \nabla_x\right)}
\newcommand{\pp}{\mathcal{P}_{\BC}}
\newcommand{\BC}{\Gamma}
\newcommand{\vn}{n}
\newcommand{\df}{\mathfrak{d}}

\begin{document}
\title[Upper bounds in a bounded domain]{Conditional $L^{\infty}$ estimates for the non-cutoff Boltzmann equation in a bounded domain}

\author{Zhimeng Ouyang and Luis Silvestre}

\thanks{Luis Silvestre is supported by NSF grants DMS-1764285 and DMS-2054888. Zhimeng Ouyang is supported by the NSF fellowship DMS-2202824.}

\address[Luis Silvestre]{Department of Mathematics, University of Chicago, Chicago, Illinois 60637, USA}
\email{luis@math.uchicago.edu}

\address[Zhimeng Ouyang]{Department of Mathematics, University of Chicago, Chicago, Illinois 60637, USA}
\email{ouyangzm9386@uchicago.edu}

\begin{abstract}
We consider weak solutions of the inhomogeneous non-cutoff Boltzmann equation in a bounded domain with any of the usual physical boundary conditions: in-flow, bounce-back, specular-reflection and diffuse-reflection. When the mass, energy and entropy densities are bounded above, and the mass density is bounded away from vacuum, we obtain an estimate of the $L^\infty$ norm of the solution depending on the macroscopic bounds on these hydrodynamic quantities only. It is a regularization effect in the sense that the initial data is not required to be bounded.
\end{abstract}
\maketitle


\section{Introduction}


We consider the spatially inhomogeneous Boltzmann equation
\begin{align}\label{eqn:boltzmann}
    \trp f=\Qff,
\end{align}
where $Q$ is the Boltzmann collision operator in the \emph{non-cutoff} case (see Subsection~\ref{SubSec:Collision-opt} for its precise formula). 

The purpose of this article is to obtain upper bounds in $L^\infty$ for the solution $f$ of \eqref{eqn:boltzmann}, in a bounded domain, that depend only on the macroscopic hydrodynamic quantities associated to $f$.

The function $f(t,x,v)$ takes values for $t \in (0,T)$, $x \in \overline \Omega$ and $v \in \R^d$ ($d\geq 2$).
The domain $\Omega \subset \R^d$ is a bounded open set with a $C^{1,1}$ boundary. We write $\BC := (0,T) \times \partial \Omega \times \R^d$, and $\BC_-$ denotes the incoming part of this boundary: $\BC_- := \{(t,x,v) \in \BC: v\cdot n(x)<0\}$ where $n(x)$ is the outward unit normal vector at $x\in\p\Omega$.

There are four common types of boundary conditions (here $\vn$ denotes the outward unit normal vector at $x\in\p\Omega$):
\begin{enumerate}
    \item 
    In-flow: $f(t,x,v)|_{\BC_-} = g(t,x,v)$ for some given function $g$;
    \item
    Bounce-back: $f(t,x,v) = f(t,x,-v)$ on $\BC$;
    \item
    Specular-reflection: $f(t,x,v) = f(t,x,\mathcal{R}_xv)$ on $\BC$, where $\mathcal{R}_xv:=v-2(v\cdot \vn)\vn$;
    \item
    Diffuse-reflection:
    $f(t,x,v)|_{\BC_-} = \pp[f](t,x,v):=c_\mu\mu(v)\int_{v'\cdot \vn>0}f(t,x,v')(v'\cdot \vn)\,\ud v'$, where $\mu(v)=e^{-\abs{v}^2}$ is the wall Maxwellian and the constant $c_\mu$ satisfies the normalization condition $c_\mu\int_{v'\cdot \vn>0}\mu(v')(v'\cdot \vn)\,\ud v'=1$.
\end{enumerate}

We work with solutions to the Boltzmann equation \eqref{eqn:boltzmann} whose macroscopic quantities are assumed to satisfy certain macroscopic bounds. More precisely, we assume that there are constants $m_0,M_0,E_0$ and $H_0$ such that
for all (or almost all) $(t,x) \in \R_+ \times \overline\Omega$ it holds that 
\begin{equation}\label{assumption:hydrodynamic}
	\begin{cases}
    0<m_0\leq\int_{\R^d}f(t,x,v)\,\ud v\leq M_0,\\
    \int_{\R^d}\abs{v}^2f(t,x,v)\,\ud v\leq E_0,\\
    \int_{\R^d}f(t,x,v)\log f(t,x,v)\,\ud v\leq H_0. 
    \end{cases}
\end{equation}
Needless to say, the inequalities \eqref{assumption:hydrodynamic} are currently unprovable for general solutions. Currently, there is arguably no reason to suspect that they hold for all solutions. The hydrodynamic bounds \eqref{assumption:hydrodynamic} are a way to say that we look at solutions that do not have implosion singularities, and by ruling them out, we argue that no other kind of singularities can exist. Our main goal would be to determine if the hydrodynamic bounds \eqref{assumption:hydrodynamic} imply the smoothness of solutions to the Boltzmann equation \eqref{eqn:boltzmann} in a bounded domain, for all possible physical boundary conditions.

\subsection{Boltzmann Collision Operator} \label{SubSec:Collision-opt}

The Boltzmann collision operator takes the form
\begin{align} \label{Boltzmann-opt}
\Qff(v)=\int_{\R^d}\int_{\S^{d-1}}B(v-\vs,\sigma)\left(f(\vs')f(v')-f(\vs)f(v)\right)\ud\sigma\ud \vs.
\end{align}
where $v'$ and $v'_\ast$ are
\begin{align*}
v'=\frac{v+\vs}{2}+\frac{\abs{v-\vs}}{2}\sigma,\qquad \vs'=\frac{v+\vs}{2}-\frac{\abs{v-\vs}}{2}\sigma.
\end{align*}
The collision kernel $B(v-\vs,\sigma)=B\left(\abs{v-\vs},\cos\theta\right)$ only depends on $\abs{v-\vs}$ and the angle $\theta=\arccos\left(\frac{v-\vs}{\abs{v-\vs}}\cdot\sigma\right)\in\left[-\pi,\pi\right]$. 

We consider standard non-cutoff collision kernels satisfying the bounds
\begin{align} \label{B-kernel}
B\left(\abs{v-\vs},\cos\theta\right)\approx \abs{v-\vs}^{\gamma}b(\cos\theta) ,
\end{align}
where $b$ is a nonnegative even function such that as $\theta\rt0$,
\begin{align} \label{b-kernel}
b(\cos\theta)\approx |\theta|^{-d+1-2s}.
\end{align}
Note that $b(\cos\theta)$ is not integrable on $\sigma \in S^{d-1}$ near $\theta=0$. Typically, this is called the non-cutoff model.

Depending on the values of $\gamma$, it is customary to use the following terminology:
\begin{itemize}
\item
Hard potentials: $\gamma > 0$.
\item
Maxwell molecules: $\gamma=0$.
\item
Moderately soft potential: $\gamma<0$ and $2s+\gamma \geq 0$.
\item
Very soft potential: $2s+\gamma<0$.
\end{itemize}
The methods in this paper work for hard or moderately soft potential. Our proof fails in the very soft potential case.

\subsection{Main Result}

The following is our main result.

\begin{theorem} \label{thm:main} 
Let $f(t,x,v)$ be a (weak) solution of the Boltzmann equation \eqref{eqn:boltzmann} for $(t,x,v)\in(0,T)\times\Omega\times\R^d$ with any of the four boundary conditions (in-flow, bounce-back, specular-reflection and diffuse-reflection). 
We consider the non-cutoff collision kernel with parameters $\gamma$ and $s$ such that 
$0<s<1$, $-d<\gamma \leq 2$, and $\gamma+2s>0$ (which covers the hard potential and moderately soft potential cases). Assume in addition that \eqref{assumption:hydrodynamic} holds for every $(t,x) \in (0,T) \times \overline\Omega$. Then we have
\[f(t,x,v)\leq a(t),\]
almost everywhere, where
\[a(t) = C \left(1+t^{-\frac{d}{2s}}\right).\]
For the bounce-back, specular-reflection and diffuse-reflection boundary, $C$ depends on the parameters $m_0$, $M_0$, $E_0$ and $H_0$ in \eqref{assumption:hydrodynamic} and dimension $d$. 
For the in-flow boundary, besides the above $(m_0,M_0,E_0,H_0,d)$ dependence, $C$ also depends on the boundary data $g$.
\end{theorem}

The program of conditional regularity has been successfully carried out for solutions to the inhomogeneous Boltzmann equation which are periodic in space (essentially removing any potential spatial boundary effect). See \cite{silvestre2016,imbert-silvestre-whi2020,imbert-mouhot-silvestre-decay2020,imbert-mouhot-silvestre-lowerbound2020,imbert-silvestre-schauder,HST-masspushing,imbert-silvestre-2022,cameron-snelson}, and the survey \cite{imbert-silvestre-survey2020}. It has also been carried out for the inhomogeneous Landau equation without spatial boundary in \cite{cameron-silvestre-snelson-2018,golse-imbert-mouhot-vasseur-2019,hendersson-snelson-2020,henderson-snelson-tarfulea-2019,snelson-2020}. While the focus of this topic is on solutions that are far from equillibrium, an incidental consequence is the global existence of smooth solutions for initial data that is close to a Maxwellian with respect to any norm for which local-in-time well posedness holds \cite{silvestre-snelson-2023}.

The upper bound provided by the main result of this paper represents a first step toward a conditional regularity program for the Boltzmann equation in a bounded domain. Recent results for second order kinetic equations in bounded domains \cite{silvestre2022,zhu2022regularity} suggest that we should expect conditional regularity estimates to hold up to the boundary at least up to the H\"older continuity of $f$. The possibility of higher regularity estimates up to the boundary requires further investigation.

For the equation without boundary, the conditional upper bounds similar to Theorem \ref{thm:main} are obtained in \cite[Theorem 1.2]{silvestre2016}. The proof there is based on a nonlocal quantitative maximum principle, following an idea similar to those used for H\"older continuity of nonvariational parabolic integro-differential equations as in \cite{silvestre-HJ,silvestre-icm2014}. It relies on a pointwise estimate of the equation at the hypothetical first crossing point between the intended upper bound and the solution. There are some technical difficulties to carry out this computation when the first crossing point is on the boundary, and because of that it does not adapt to the setting of this paper easily. We opt for a different, more variational, approach, arguably similar to other techniques developed originally for parabolic integro-differential equations in \cite{caffarelli2011}. As a byproduct of this approach, the result applies to a more general class of weak solutions. In that sense, Theorem \ref{thm:main} provides a small improvement over \cite[Theorem 1.2]{silvestre2016}, even for the case without boundary, that can make the result easier to apply. The notion of solution that we use in this paper is descibed in Section \ref{Sec:weak-sol}, and it is a more or less classical weak solution in Sobolev spaces. It is not as weak as the renormalized solutions with defect measure constructed in \cite{alexandre-villani}, but it is certainly weaker than the classical solutions considered in \cite{silvestre2016}.

It is worth noting that while we make some technical assumption on the solution (belonging to certain Sobolev space) and on the domain (bounded and with a $C^{1,1}$ boundary), these are qualitative conditions simply to be able to make sense of all the steps in the proof. Our final estimate does not depend in any way to any quantity related to these technical assumptions.

\subsubsection{On the Diffuse-Reflection Boundary Condition} \label{s:diffuse}

The diffuse-reflection boundary condition can be seen as a particular case of the in-flow boundary condition. Instead of having an arbitrary value for $f$ on $\BC_-$, its value is related to the values of $f$ on $\BC_+$ by some averaging procedure.

The hydrodynamic assumptions \eqref{assumption:hydrodynamic} trivially imply an upper bound for $f_{\BC_-} = \pp[f](t,x,v)$ in the case of diffuse-reflection boundary. Indeed,
\[ \abs{ \int_{v'\cdot \vn>0}f(t,x,v')(v'\cdot \vn)\,\ud v' } \leq \int_{\R^d }f(t,x,v') |v'| \,\ud v' \leq \sqrt{M_0 E_0}.\]
Therefore, Theorem \ref{thm:main} for the diffuse-reflection boundary is a corollary of the same theorem for the in-flow boundary condition.

After this observation, we will not make any further reference to the diffuse-reflection boundary condition in the rest of the paper, except for Remark \ref{r:diffuse} about the notion of weak solutions.

\subsection*{Acknowledgements}

The authors would like to thank Lei Wu for many helpful discussions.

\section{Preliminaries}

\subsection{Notation and Convention}

Throughout this paper, $C$ will generally denote a universal constant which may depend on $d,m_0,M_0,E_0,H_0$.
We write $A \lesssim B$ whenever $A\leq CB$ for some universal constant $C>0$; we will use $\gtrsim$ and $\approx$ in a similar standard way.

We split $\BC= (0,T)\times\partial\Omega\times\R^d$ into the outgoing boundary $\BC_+$, the incoming boundary $\BC_-$, and the singular boundary $\BC_0$ (i.e., the ``grazing set''):
\begin{equation*}
\begin{split}
\BC_+ &:= \left\{(t,x,v)\in\BC: v\cdot \vn>0\right\} ,\\
\BC_- &:= \left\{(t,x,v)\in\BC: v\cdot \vn<0\right\} ,\\
\BC_0 &:= \left\{(t,x,v)\in\BC: v\cdot \vn=0\right\} .
\end{split}
\end{equation*}
Sometimes we use $\ud\BC$ to denote $\ud v\ud S_x\ud t$, where $\ud S_x$ is the differential of surface for $x\in\Omega$.
 
We use $L_n^p$ to denote the weighted $L^p$ norm
\[ \nm{f}_{L_n^p(\R^d)}:=\left(\int_{\R^d}\br{v}^{np}\abs{f(v)}^p\,\ud v\right)^{\frac{1}{p}}. \]
When $n=0$, it reduces to the usual $L^p$ norm.

\subsection{Weighted Kinetic Sobolev Space} \label{Subsec:kinetic-space}

In order to make sense of the hydrodynamic bounds \eqref{assumption:hydrodynamic}, we naturally need to start with a solution $f$ that belongs to the space $L^\infty((0,T)\times \Omega, L^1_2(\R^d)) \cap L^\infty((0,T)\times \Omega, L\log L(\R^d))$. In order to define the notion of weak solution for which our result applies, we impose some further technical condition in terms of a kinetic Sobolev space.

Following \cite{gressmanstrain2011}, we define the weighted (anisotropic) Sobolev space $N^{s,\gamma}$ as the space of functions $f : \R^d \to \R$ for which the following norm is bounded.
\begin{align} \label{eq:normNsg}
    \nm{f}_{N^{s,\gamma}}^2:=\int_{\R^d}\br{v}^{\gamma+2s}\abs{f(v)}^2\ud v+\iint_{\R^d\times\R^d}\br{v}^{\frac{\gamma+2s+1}{2}}\br{v'}^{\frac{\gamma+2s+1}{2}}\frac{\abs{f(v)-f(v')}^2}{\df(v,v')^{d+2s}}\mathds{1}_{\df(v,v')\leq 1}\,\ud v'\ud v,
\end{align}
where 
\[
    \df(v,v'):=\sqrt{\abs{v-v'}^2+\tfrac{1}{4}\left(\abs{v}^2-\abs{v'}^2\right)^2}.
\]
The $N^{s,\gamma}$ norm is induced by the inner product
\begin{align*}
    \br{f,g}_{N^{s,\gamma}}:=&\int_{\R^d}\br{v}^{\gamma+2s}f(v)g(v)\,\ud v\\
    &+\iint_{\R^d\times\R^d}\br{v}^{\frac{\gamma+2s+1}{2}}\br{v'}^{\frac{\gamma+2s+1}{2}}\frac{\left(f(v)-f(v')\right)\left(g(v)-g(v')\right)}{\df(v,v')^{d+2s}}\one_{\df(v,v')\leq 1}\,\ud v'\ud v.
\end{align*}
Denote $N^{-s,\gamma}$ as the dual of $N^{s,\gamma}$ equipped with the norm
\[ \nm{g}_{N^{-s,\gamma}}:=\sup_{\nm{\varphi}_{N^{s,\gamma}}=1}\int_{\R^d} g(v) \varphi(v) \,\ud v. \]

The space $N^{s,\gamma}$ should be understood as a weighted version of $H^s(\R^d)$, with a weight that accounts for the precise behavior of the collision operator $\Qff$ as $|v| \to \infty$. Likewise, the space $N^{-s,\gamma}$ is a weighted version of the space $H^{-s}(\R^d)$. 

\begin{lemma}\label{lem:negative-norm}
    Consider the following integro-differential operator
    \[ \mathcal{L}f(v):=\br{v}^{\gamma+2s}f(v)+\int_{\R^d}\br{v}^{\frac{\gamma+2s+1}{2}}\br{v'}^{\frac{\gamma+2s+1}{2}}\frac{2\left(f(v)-f(v')\right)}{\df(v,v')^{d+2s}}\mathds{1}_{\df(v,v')\leq 1}\,\ud v'. \]
    Then $\mathcal L$ maps $N^{s,\gamma}$ into $N^{-s,\gamma}$. Moreover, for every $g \in N^{-s,\gamma}$, there exists $f\in N^{s,\gamma}$ such that $g = \mathcal L f$.
\end{lemma}
\begin{proof}
This is a direct consequence of the Riesz representation theorem applied to the Hilbert space $N^{s,\gamma}$. Indeed, by a direct computation we observe that whenever $f$ is smooth enough to make sense of the integral in the definition of $\mathcal Lf$, we have
\[ \br{f,g}_{N^{s,\gamma}} = \int \mathcal{L}f(v) g(v) \,\ud v. \]
\end{proof}

The operator $\mathcal L$ defined in Lemma \ref{lem:negative-norm} is similar to $I + (-\Delta)^s$, but with a weight for large $|v|$ matching that in the definition of $N^{s,\gamma}$.

Following \cite{albritton2019variational}, we define the weighted kinetic Sobolev space $\Hk$ as follows.
\begin{definition} \label{Def:kinetic-space}
Given an open set $D:=(0,T)\times\Omega\times\R^d \subset \R^{1+2d}$, we say $f \in \Hk(D)$ if $f \in L_{t,x}^2N_v^{s,\gamma}(D)$ and $\trp f \in L^2_{t,x} N_v^{-s,\gamma}(D)$ in the sense that
\[ \int_D f(t,x,v) \trp \varphi \,\ud v\ud x\ud t \leq C \| \varphi \|_{L_{t,x}^2N_v^{s,\gamma}(D)} \]
for any function $\varphi \in C^1_c(D)$ with $\text{supp}(\varphi)\subset D$. 
We further write
\[\|f\|_{\Hk(D)}^2 := \|f\|_{L_{t,x}^2N_v^{s,\gamma}(D)}^2 + \nm{\trp f}_{L^2_{t,x} N_v^{-s,\gamma}(D)}^2.\]
\end{definition}

The following result is proved essentially in \cite[Theorem 2]{gressmanstrain2011}.

\begin{theorem} \label{t:boundednessofQ}
Given a function $f: \R^d \to [0,\infty)$ that satisfies
\[ \int_{\R^d} f(v) \,\ud v \leq M_0, \qquad \int_{\R^d} f(v) |v|^2 \,\ud v \leq E_0.\]
Let $g$ and $h$ be two functions in $N^{s,\gamma}$, then
\[ \int_{\R^d} Q(f,g)h \,\ud v \leq C_f \nm{g}_{N^{s,\gamma}} \nm{h}_{N^{s,\gamma}}.\]
\end{theorem}

The Assumption U, given in \cite{gressmanstrain2011} is sligtly different and not directly implied by our upper bounds on $M_0$ and $E_0$. However, this version of Theorem \ref{t:boundednessofQ} follows with minimal effort after the ideas developed in \cite{gressmanstrain2011} (see also \cite[Theorem 4.1]{imbert-silvestre-whi2020} and Appendix A of \cite{imbert-silvestre-2022}).

A way to interpret Theorem \ref{t:boundednessofQ} is that if $f$ satisfies \eqref{assumption:hydrodynamic} and $g \in L^2_{t,x} N_v^{s,\gamma}$, then $Q(f,g) \in L^2_{t,x} N_v^{-s,\gamma}$, with
\begin{align}\label{q-estimate}
    \nm{Q(f,g)}_{L^2_{t,x} N_v^{-s,\gamma}} \leq C_f \nm{g}_{L^2_{t,x} N_v^{s,\gamma}}.
\end{align}

\subsection{Notion of Weak Solutions}
\label{Sec:weak-sol}

Now that we described the functional spaces that we use, we define the notion weak solutions in $D=(0,T) \times \Omega \times \R^d$ with the four types of boundary conditions.

\begin{definition}[Weak solutions with in-flow boundary] \label{d:weak-sol-inflow}
We say that a function $f \in L_{t,x}^2N_v^{s,\gamma}(D)$ is a weak solution of \eqref{eqn:boltzmann} in $D$ with the in-flow boundary condition,
if for any test function $\varphi \in C^1_c\left((0,T) \times \overline\Omega \times \R^d\right)$ so that $\varphi|_{\BC_+}=0$, we have
\begin{align}\label{weak-formulation-inflow}
&\iiint_{(0,T)\times\Omega\times\R^d} \left\{f \,\trp \varphi + \Qff \varphi \right\}\,\ud v\ud x\ud t=\iiint_{\BC_-} g \varphi \,(v \cdot \vn) \ud v\ud S_x\ud t. 
\end{align}
\end{definition}

We recall that $\Qff$ belongs to $L^2_{t,x} N^{-s,\gamma}_v$ according to \eqref{q-estimate}. In this sense the integral of $\Qff \varphi$ is well defined for any test function $\varphi \in C^1_c\left((0,T) \times \overline\Omega \times \R^d\right)$.

\begin{definition}[Weak solutions with bounce-back boundary] \label{d:weak-sol-bounce-back} 
We say that a function $f \in L_{t,x}^2N_v^{s,\gamma}(D)$ is a weak solution of \eqref{eqn:boltzmann} in $D$ with the bounce-back boundary condition,
if for any test function $\varphi \in C^1_c\left((0,T) \times \overline\Omega \times \R^d\right)$ so that $\varphi(t,x,v)=\varphi(t,x,-v)$ for all $x\in\partial\Omega$, we have
\begin{align}  \label{weak-formulation-bounceback}
&\iiint_{(0,T)\times\Omega\times\R^d} \left\{f \,\trp \varphi + \Qff \varphi \right\}\,\ud v\ud x\ud t=0.
\end{align}
\end{definition}

\begin{definition}[Weak solutions with specular-reflection boundary] \label{d:weak-sol-specular}
We say that a function $f \in L_{t,x}^2N_v^{s,\gamma}(D)$ is a weak solution of \eqref{eqn:boltzmann} in $D$ with the specular-reflection boundary condition,
if for any test function $\varphi \in C^1_c\left((0,T) \times \overline\Omega \times \R^d\right)$ so that $\varphi(t,x,v)=\varphi(t,x,\mathcal{R}_xv)$ for all $x\in\partial\Omega$, we have
\begin{align} \label{weak-formulation-specular}
&\iiint_{(0,T)\times\Omega\times\R^d} \left\{f \,\trp \varphi + \Qff \varphi \right\}\,\ud v\ud x\ud t=0.
\end{align}
\end{definition}

\begin{remark} \label{r:diffuse} 
As described in Section \ref{s:diffuse}, we prove Theorem \ref{thm:main} in the cases of in-flow, bounce-back and specular-reflection boundary conditions. The case of diffuse-reflection is a particular case of the inflow boundary condition. If we wanted to define a weak solution in this case, it would coincide with Definition~\ref{d:weak-sol-inflow} but with $\pp[f]$ instead of the arbitrary function $g$. In any context where it makes sense to consider the diffuse-reflection boundary condition, the analysis in Section \ref{s:diffuse} holds and Theorem \ref{thm:main} applies.

It may have some interest to observe that the operator $\pp[f]$ is not necessarily well defined for arbitrary functions $f \in \Hk$. However, when $f$ satisfies in addition the assumption \eqref{assumption:hydrodynamic}, then we can easily see that the trace function $f|_{\BC}$ obtained from Proposition \ref{p:trace} belongs to $L^\infty((0,T)\times \partial \Omega, L^1_2(\R^d))$. Indeed, the upper bounds on the mass and energy densities in \eqref{assumption:hydrodynamic} tell us that $f \in L^\infty((0,T)\times \Omega, L^1_2(\R^d))$. We can rigorously justify that the trace of $f$ satisfies the same bound by approximating $f$ with smooth functions (as in Lemma~\ref{lem:mollification}), observing that the trace of the approximate functions are bounded in $L^\infty((0,T)\times \partial \Omega, L^1_2(\R^d))$, and applying Fatou's lemma.
\end{remark}

\begin{remark}
Definitions \ref{d:weak-sol-inflow}, \ref{d:weak-sol-bounce-back} and \ref{d:weak-sol-specular} encode at the same time the equation \eqref{eqn:boltzmann} in the interior of the domain and the boundary conditions. Indeed, taking the test function $\varphi$ to be compactly supported in the interior of $D$, we deduce for any of these definitions that $\trp f = \Qff$ holds in the sense of distributions in $D$. After this, replacing $\Qff$ with $\trp f$ in Definition \ref{d:weak-sol-inflow}, we observe that for any test function $\varphi \in C^1_c\left((0,T) \times \overline\Omega \times \R^d\right)$ so that $\varphi|_{\BC_+}=0$, we have
\[ \int_0^T\iint_{\Omega\times\R^d} \left\{\trp f \,\varphi + f \,\trp \varphi \right\}\,\ud v\ud x\ud t=\int_0^T\iint_{\BC_-} g \varphi \,(v \cdot \vn) \ud v\ud S_x\ud t. \]
This identity encodes the boundary condition (without the rest of the equation). A similar identity follows for each type of boundary condition.
\end{remark}

\subsection{Structure of the Collision Operator} \label{Sec:opt-str}

Recall the non-cutoff Boltzmann collision operator $\Qff$ from \eqref{Boltzmann-opt} with the parameters given in \eqref{B-kernel} and \eqref{b-kernel}: $2s$ is the angular singularity exponent and $\gamma$ is the exponent in terms of $|v-v_\ast|$.

Let us use Carleman coordinates
\begin{align*}
    (\sigma,\vs)\rt \left(w:=\vs'-v,v'\right), \\
    w\perp (v'-v),\qquad \vs=v'+w.
\end{align*}
In terms of these new variables, the collision operator becomes 
\begin{align}\label{kernel}
    \Qff(v)&=\int_{\R^d}\left(f(v')K_f(v,v')-f(v)K_f(v',v)\right)\ud v'\\
    &=\int_{\R^d}\left(f(v')-f(v)\right)K_f(v,v')\,\ud v'+ 
    f(v)\int_{\R^d}\left(K_f(v,v')-K_f(v',v)\right)\ud v',\no
\end{align}
where the kernel $K_f$ depends on $f$ through
\begin{align*}
    K_f(v,v')&=2^{d-1}\abs{v'-v}^{-1}\int_{w\perp(v'-v)}f(v+w)B(r,\cos\theta) r^{-d+2}\,\ud w\ \ \text{ with}\ \ \begin{cases}r^2=\abs{v'-v}^2+\abs{w}^2\\
    \cos\left(\frac{\theta}{2}\right)=\frac{\abs{w}}{r}\end{cases}\\
    &\approx \abs{v-v'}^{-d-2s}\left\{\int_{w\perp(v'-v)}f(v+w)\abs{w}^{\gamma+2s+1}\ud w\right\}.
\end{align*}
Furthermore, the well-known cancellation lemma takes the following form in terms of the kernel $K_f$.
\begin{align}\label{K-cancellation}
    \int_{\R^d}\left(K_f(v,v')-K_f(v',v)\right)\ud v'=c_b\int_{\R^d}f(v_*)\abs{v-v_*}^{\gamma}\ud v_*,
\end{align}
with constant 
\begin{align*}
    c_b=\int_{\S^{d-1}}\left\{\frac{2^{\frac{d+\gamma}{2}}}{(1+\sigma\cdot\mathbf{e})^{\frac{d+\gamma}{2}}}-1\right\}b(\sigma\cdot\mathbf{e})\,\ud\sigma
\end{align*}
for any $\mathbf{e}\in\S^{d-1}$. 
Therefore, we may rewrite
\begin{align} \label{Q-diffusion}
    \Qff&=\int_{\R^d}\left(f(v')-f(v)\right)K_f(v,v')\,\ud v'+ 
    c_b\left(\int_{\R^d}f(v_*)\abs{v-v_*}^{\gamma}\ud v_*\right) f(v)\\
    &=:\li_Kf+c_b\left(f\ast_v\abs{\,\cdot\,}^{\gamma}\right)f.\no
\end{align}
Here, $\li_K$ is a nonlinear integro-differential diffusion operator which leads to smoothing eﬀect, and $c_b\left(f\ast_v\abs{\,\cdot\,}^{\gamma}\right)f$ is a lower-order term.
One may think of $\Qff$ as a reaction diffusion operator. The second term $c_b\left(f\ast_v\abs{\,\cdot\,}^{\gamma}\right)f$ would be the reaction term which makes the function $f$ grow, and it is therefore the bad term when it gets to compute $L^\infty$ estimates.

The important properties of the diffusion kernel $K_f$ are the following (see \cite{imbert-silvestre-survey2020,silvestre2016}):
\begin{itemize}
	\item Symmetry: $K_f(v,v+u) = K_f(v,v-u)$ for any $u\in\R^d$.
	\item Bounded from above by the fractional Laplacian on average:
	\begin{align} \label{K-upper-bound}
	    \int_{B_r(v)} K_f(v,v') |v-v'|^2 \,\ud v' \leq C \left(\int_{\R^d} f(v_\ast) |v-v_\ast|^{\gamma+2s} \,\ud v_\ast\right) r^{2-2s}
	    \leq \Lambda(v)\, r^{2-2s} .
	\end{align}
	Here $\Lambda(v) = C \langle v \rangle^{\gamma+2s}$ for some constant $C>0$ depending on $M_0$ and $E_0$ in \eqref{assumption:hydrodynamic}.
	\item Cone of non-degeneracy: 
	for fixed $t$ and $x$,
	there exists a set $\Xi(v) \subset \R^d$ for every point $v \in \R^d$ so that
	\begin{itemize}
		\item $\Xi(v)$ is a symmetric cone, which means that $\lambda \Xi(v) = \Xi(v)$ for all $\lambda \in \R$, $\lambda \neq 0$.
		\item Lower bound for $K_f$ in the directions of $\Xi(v)$: 
		\begin{align} \label{K-lower-bound}
		    K_f(v,v') \geq \lambda(v) |v-v'|^{-d-2s}
		    \;\;\text{ whenever }\; v'-v \in \Xi(v).
		\end{align}
		Here $\lambda(v) = c \langle v \rangle^{\gamma+2s+1}$ for some constant $c>0$ depending on 
		the parameters $m_0$, $M_0$, $E_0$ and $H_0$ in \eqref{assumption:hydrodynamic} and dimension $d$.
		\item The measure $m\left(\Xi(v) \cap \S^{d-1}\right) \geq c \langle v \rangle^{-1}$, for some constant $c$ depending on $m_0$, $M_0$, $E_0$ and $H_0$ in \eqref{assumption:hydrodynamic} and dimension $d$. Moreover, $\Xi(v) \cap \S^{d-1}$ is contained in a band of width $\lesssim \langle v \rangle^{-1}$ around the equator perpendicular to $v$.
	\end{itemize}
	\item Cancellation lemma \eqref{K-cancellation}.
\end{itemize}
\begin{remark}
Note that the bounds \eqref{K-upper-bound} and \eqref{K-lower-bound} are weaker than the usual uniform ellipticity condition $\lambda|v-v'|^{-d-2s}\leq K(v,v')\leq \Lambda|v-v'|^{-d-2s}$ for integro-differential equations. The lower bound holds only in the cone of nondegeneracy, and the upper bounds holds only in average.

We also stress that the existence of the cone of non-degeneracy of $K_f$ relies on the hydrodynamic bounds in \eqref{assumption:hydrodynamic}, and this is the only place where the entropy bound $H_0$ is used.
\end{remark}

\subsection{Coercivity Estimate}

When computing the propagation of $L^2$ norms of various functions, one often deals with the integral expressions from \eqref{Q-diffusion} of the form
\begin{align} \label{Q-fgh}
    \int_{\R^d}Q(f,g)(v)h(v)\,\ud v=\iint_{\R^d\times\R^d}h(v)\left(g(v')-g(v)\right)K_f(v,v')\,\ud v'\ud v+c\int_{\R^d}\left(f\ast_v\abs{v}^{\gamma}\right)g(v)h(v)\,\ud v.
\end{align}
In particular, when $h=g$, we have
\begin{align} \label{Q-fgg}
    \int_{\R^d}Q(f,g)(v)g(v)\,\ud v=-\frac{1}{2}\iint_{\R^d\times\R^d}\abs{g(v')-g(v)}^2K_f(v,v')\,\ud v'\ud v+c\int_{\R^d}\left(f\ast_v\abs{v}^{\gamma}\right)|g(v)|^2\,\ud v.
\end{align}
This identity follows by a straightforward arithmetic manipulation and applying Fubini’s theorem. 
The first term is negative. Coercivity estimates for the Boltzmann collision operator amount to estimating how negative this first term needs to be.

The following inequality may be seen as a combination of a coercivity estimate with the Sobolev embedding for weighted Sobolev norms. Except that it is more complicated to prove each one of these inequalities separately. Here we verify the formula in one shot, using a computation similar to \cite{chaker2022entropy}.

\begin{lemma}[Coercivity estimate] \label{lem:coercivity}
Let $p>2$ be the exponent satisfying $\frac{1}{p}=\frac{1}{2}-\frac{s}{d}$,
and let $n=\frac{1}{2}\left(\gamma+2s-\frac{2s}{d}\right)$ and $k=\frac{1}{2}\left(-\gamma-d+1\right)$.
Then there exist constants $c_0>0$ and $C_1\geq 0$ (depending only on the hydrodynamic bounds in \eqref{assumption:hydrodynamic} and dimension $d$) such that
for any $g\in L_n^p(\R^d)$, it holds that
\begin{align} \label{coercivity_n<0}
    \iint_{\R^d\times\R^d}\abs{g(v')-g(v)}^2K_f(v,v')\,\ud v'\ud v\geq c_0\nm{g}_{L_n^p(\R^d)}^{2-p}
    \cdot\! \int_{\left\{v\in\R^d:\abs{g(v)}\geq C_1\nm{g}_{L_n^p}\br{v}^k\right\}} \langle v \rangle^{np} |g(v)|^p \,\ud v.
\end{align}
In particular, when $n\geq 0$, we have a stronger bound 
\begin{align} \label{coercivity_n>0}
    \iint_{\R^d\times\R^d}\abs{g(v')-g(v)}^2K_f(v,v')\,\ud v'\ud v\geq c_0\nm{g}_{L_n^p(\R^d)}^2,
\end{align}
which is equivalent to taking $C_1=0$ in \eqref{coercivity_n<0}.
\end{lemma}

\begin{proof}
Let us denote
\[ N := \nm{g}_{L_n^p(\R^d)}^p = \int_{\R^d} \langle v \rangle^{np} |g(v)|^p \,\ud v <\infty. \]
For any fixed $v\in\R^d$ (such that $g(v)\neq0$), 
we will exploit the cone of non-degeneracy $\Xi(v)$ for $K_f$ where $K_f$ has a lower bound (see Subsection~\ref{Sec:opt-str}).
Recall that the intersection of the cone $\Xi(v)$ with a ball $B_R$ has volume $\approx R^d \langle v \rangle^{-1}$.
We choose $R=R(v)$ (depending on $v$) satisfying 
\begin{align} \label{e:R}
R^d = CN \langle v \rangle^{-np+1}|g(v)|^{-p} 
\end{align} 
for some large constant $C>0$.
Next we split $\left\{v\in\R^d:g(v)\neq0\right\}$ into two sets: 
\begin{align*}
    G&:= \left\{v\in\R^d: \br{v}\geq R(v) \right\} 
    = \left\{\abs{g(v)}^p\geq CN\br{v}^{-d-np+1}\right\}
    = \left\{\abs{g(v)}\geq C^{\frac{1}{p}}\nm{g}_{L_n^p}\br{v}^{k}\right\}, \\
    B&:= \left\{v\in\R^d: \br{v}< R(v) \right\} 
    = \left\{\abs{g(v)}< C^{\frac{1}{p}}\nm{g}_{L_n^p}\br{v}^{k}\right\}.
\end{align*}

For $v\in G$, we claim that $|g(v')| < \frac{1}{2}|g(v)|$ for points $v'$ in $v + \left(B_{R/2} \cap \Xi(v)\right)$ that amount to at least half of its measure, i.e. 
\begin{align} \label{xiaofu}
    \abs{\left\{v'\in v + \left(B_{R/2} \cap \Xi(v)\right): |g(v')| < \tfrac{1}{2}|g(v)| \right\}}
    \geq \tfrac{1}{2}\abs{B_{R/2} \cap \Xi(v)}\approx R^d \langle v \rangle^{-1}.
\end{align}
Indeed, since $\br{v}\geq R$ and $v'\in v+B_{R/2}$, we have $\frac{1}{2}\langle v \rangle\leq\br{v'}\leq\frac{3}{2}\langle v \rangle$.
Then the converse statement that $|g(v')| \geq \frac{1}{2}|g(v)|$ for at least half of $v'\in v + \left(B_{R/2} \cap \Xi(v)\right)$ implies 
\[ N = \int_{\R^d} \langle v' \rangle^{np} |g(v')|^p \,\ud v' \gtrsim \langle v \rangle^{np} |g(v)|^p\cdot R^d \langle v \rangle^{-1} = CN, \]
which cannot be true for sufficiently large $C$.
For $v\in B$ where $\br{v}< R$, we have the similar argument only in the case of $n\geq0$:
for $v'\in v+(B_{3R}\backslash B_{2R})$ we have $\br{v'}\geq\langle v \rangle$, and thus \eqref{xiaofu} still holds with $B_{R/2}$ replaced by $B_{3R}\backslash B_{2R}$. 

For the analysis below, we first treat the case when $v\in G$. 
Considering $K_f(v,v')\geq0$ and that by \eqref{K-lower-bound}
\[ 
	K_f(v,v') \gtrsim \langle v \rangle^{\gamma+2s+1} |v-v'|^{-d-2s}
	\;\;\text{ if }\; v'-v \in \Xi(v) ,
\]
we combine \eqref{xiaofu} to get
\begin{align} \label{temp}
    \int_{\R^d}\abs{g(v')-g(v)}^2K_f(v,v')\,\ud v'
    &\geq\int_{v + (B_{R/2} \cap \Xi(v))}\abs{g(v')-g(v)}^2K_f(v,v')\,\ud v'\\
    &\gtrsim \langle v \rangle^{\gamma+2s+1} R^{-d-2s} \int_{v + (B_{R/2} \cap \Xi(v))}\abs{g(v')-g(v)}^2\,\ud v' \no\\
    &\gtrsim \langle v \rangle^{\gamma+2s} R^{-2s} |g(v)|^2 
    = (CN)^{\frac{2}{p}-1}\langle v \rangle^{np} |g(v)|^p. \no
\end{align}
Here in the last equality we plug in our choice of $R$ in \eqref{e:R} with values of $p$ and $n$.
Finally, we integrate \eqref{temp} over $v\in G$ to obtain 
\[ \int_{G}\int_{\R^d}\abs{g(v')-g(v)}^2K_f(v,v')\,\ud v'\ud v \gtrsim N^{\frac{2}{p}-1} \int_{G} \langle v \rangle^{np} |g(v)|^p \,\ud v, \]
and so \eqref{coercivity_n<0} follows.

In the case of $n\geq0$, we can also get the same estimate \eqref{temp} for $v\in B$, by writing $B_{3R}\backslash B_{2R}$ instead of $B_{R/2}$.
For those $v\in\R^d$ such that $g(v)=0$, \eqref{temp} automatically holds, and thus it holds for all $v\in\R^d$.
Therefore, we may integrate \eqref{temp} over $v\in \R^d$ to obtain 
\begin{align*}
    \iint_{\R^d\times\R^d}\abs{g(v')-g(v)}^2K_f(v,v')\,\ud v'\ud v
    &\gtrsim N^{\frac{2}{p}-1} \int_{\R^d} \langle v \rangle^{np} |g(v)|^p \,\ud v
    = N^{\frac{2}{p}} = \nm{g}_{L_n^p(\R^d)}^2.
\end{align*}
This completes the proof of the lemma.
\end{proof}

\begin{remark}
The value of $p$ is the same exponent as in the usual Sobolev embedding $H^{s}(\R^d)\hookrightarrow L^p(\R^d)$. In fact, if we repeat the proof above for a kernel $K$ satisfying the stronger assumption $K(v,v') \gs \abs{v-v'}^{-d-2s}$ for all $v,v' \in \R^d$, we would derive the standard Sobolev inequality for fractional Sobolev spaces.
\end{remark}

\begin{remark}
The case $n=\frac{1}{2}\left(\gamma+2s-\frac{2s}{d}\right)<0$ happens only for soft potentials, and covers part of the moderately soft potential range.
It is interesting to note that this coercivity estimate also applies to the very soft potential case when $\gamma+2s\leq0$ (so $n<0$). 
\end{remark}

\section{Smooth Approximations and Trace} \label{Sec:smooth-approx}

\subsection{Some Properties of the Weighted Fractional Sobolev Spaces}

We analyze the spaces $N^{s,\gamma}$, $N^{-s,\gamma}$ and $\Hk$.

\begin{lemma} \label{l:comparability-of-sobolev-norms}
Suppose that $f \in N^{s,\gamma}$ and $f(v) = 0$ whenever $|v|>R$. Then, for some constant $C$ depending only on $s$, $\gamma$, the dimension $d$ and $R$, we have
\[ \frac 1 C \|f\|_{H^s(\R^d)} \leq \|f\|_{N^{s,\gamma}} \leq C \|f\|_{H^s(\R^d)}.\]
Here, $H^s(\R^d)$ denotes the standard fractional Sobolev space in $\R^d$ with the norm
\[ \|f\|_{H^s(\R^d)}^2 = \int_{\R^d} \abs{f(v)}^2 \ud v + \iint \frac{|f(v')-f(v)|^2}{|v'-v|^{d+2s}} \,\ud v' \ud v.\]
\end{lemma}

Lemma \ref{l:comparability-of-sobolev-norms} is a consequence of the fact that $\df(v,v') \approx |v'-v|$ if we know that $v'$ and $v$ stay in some bounded ball $B_R$. The proof of Lemma \ref{l:comparability-of-sobolev-norms} is a direct computation following this fact and we skip it.

\begin{lemma} \label{l:comparability-of-negative-sobolev-norms}
Suppose that $f \in N^{-s,\gamma}$ and $\supp f \subset B_R$ (in the sense of distributions). Then, for some constant $C$ depending only on $s$, $\gamma$, the dimension $d$ and $R$, we have
\[ \frac 1 C \|f\|_{H^{-s}(\R^d)} \leq \|f\|_{N^{-s,\gamma}} \leq C \|f\|_{H^{-s}(\R^d)}.\]
Here, $H^{-s}(\R^d)$ denotes the standard negative fractional Sobolev space which is the dual of $H^s(\R^d)$.
\end{lemma}

\begin{proof}
By definition
\[ \|f\|_{N^{-s,\gamma}} = \sup \left\{ \langle f, g \rangle : \|g\|_{N^{s,\gamma}} = 1 \right\}.\]
Here $\langle \cdot,\cdot \rangle$ denotes the pairing between $N^{-s,\gamma}$ and $N^{s,\gamma}$.

Since $\supp f \subset B_R$, if we take a smooth bump function $b : \R^d \to [0,1]$ so that $b = 1$ in $B_R$ and $b=0$ outside $B_{2R}$, we have $\langle f , b g \rangle = \langle f , g \rangle$. We have $\|bg\|_{N^{s,\gamma}} \lesssim \|g\|_{N^{s,\gamma}}$ (from Lemma \ref{lem:Nsgofproduct}). Thus, we may only lose a constant factor by restricting the supremum above to those functions supported in $B_{2R}$.
\[ \|f\|_{N^{-s,\gamma}} \approx \sup \left\{ \langle f, g \rangle : \|g\|_{N^{s,\gamma}} = 1 \text{ and } \supp g \subset B_{2R} \right\}.\]

From Lemma \ref{l:comparability-of-sobolev-norms}, we know that $\|g\|_{N^{s,\gamma}} \approx \|g\|_{H^s}$, and the result follows.
\end{proof}

\begin{lemma} \label{l:commutator}
Let $\mathcal L$ be the operator from Lemma \ref{lem:negative-norm}. For any smooth function $b : \R^d \to \R$ with compact support and $f \in N^{s,\gamma}$, we have the following commutator estimate
\[ \left(\int_{\R^d} \langle v\rangle^{-(\gamma+2s)} \abs{\mathcal L[bf](v) - b(v) \mathcal L f(v)}^2 \ud v\right)^{1/2} \leq M_b \|f\|_{N^{s,\gamma}},\]
where
\[ M_b := 2\sup_v \left\vert \int_{\df(v,v')<1} \langle v \rangle^{\frac 12}\langle v' \rangle^{\frac 12} \frac{b(v')-b(v)}{\df(v,v')^{d+2s}} \,\ud v' \right\vert + 2\sup_v \left(\int_{\df(v,v')<1} \langle v \rangle^{\frac 12}\langle v' \rangle^{\frac 12} \frac{|b(v')-b(v)|^2}{\df(v,v')^{d+2s}} \,\ud v'\right)^{1/2}.\]

Consequently,
\[ \nm{\mathcal L[bf](v) - b(v) \mathcal L f(v)}_{N^{-s,\gamma}} \leq M_b \|f\|_{N^{s,\gamma}}.\]
\end{lemma}

\begin{proof}
Computing $\mathcal L[bf](v) - b(v) \mathcal L f(v)$ directly using the formula of Lemma \ref{lem:negative-norm}, many terms cancel out and we are left with
\begin{align*}
\mathcal L[bf](v) - b(v) \mathcal L f(v) &= \int_{\df(v,v')<1} \langle v \rangle^{\frac{\gamma+2s+1}2}\langle v' \rangle^{\frac{\gamma+2s+1}2} \frac{f(v') \left(b(v)-b(v')\right)}{\df(v,v')^{d+2s}} \,\ud v' \\
&= \int_{\df(v,v')<1} \langle v \rangle^{\frac{\gamma+2s+1}2}\langle v' \rangle^{\frac{\gamma+2s+1}2} \frac{\left(f(v')-f(v)\right) \left(b(v)-b(v')\right)}{\df(v,v')^{d+2s}} \,\ud v'  \\
&\phantom{=} + f(v) \int_{\df(v,v')<1} \langle v \rangle^{\frac{\gamma+2s+1}2}\langle v' \rangle^{\frac{\gamma+2s+1}2} \frac{b(v)-b(v')}{\df(v,v')^{d+2s}} \,\ud v' \\
\intertext{We use the Cauchy–Schwarz inequality for the first term.}
\abs{\mathcal L[bf](v) - b(v) \mathcal L f(v)} &\lesssim \langle v \rangle^{\frac{\gamma+2s}2} M_b \left( \int_{\df(v,v')<1} \langle v \rangle^{\frac{\gamma+2s+1}2}\langle v' \rangle^{\frac{\gamma+2s+1}2} \frac{|f(v')-f(v)|^2}{\df(v,v')^{d+2s}} \,\ud v' \right)^{1/2} + \langle v \rangle^{\gamma+2s} M_b f(v).
\end{align*}
Taking squares and integrating in $v$, we conclude
\begin{align*}
\int_{\R^d} \langle v \rangle^{-(\gamma+2s)} \abs{\mathcal L[bf](v) - b(v) \mathcal L f(v)}^2 \,\ud v &\lesssim M_b^2 \left( \iint_{\df(v,v')<1} \langle v \rangle^{\frac{\gamma+2s+1}2}\langle v' \rangle^{\frac{\gamma+2s+1}2} \frac{|f(v')-f(v)|^2}{\df(v,v')^{d+2s}} \,\ud v' \ud v \right) \\ &\qquad + M_b^2 \int \langle v \rangle^{\gamma+2s}  f(v)^2 \,\ud v \\
&\leq M_b^2 \|f\|_{N^{s,\gamma}}^2.
\end{align*}
The last inequality in Lemma \ref{l:commutator} follows from the observation that since for any function $f \in N^{s,\gamma}$,
\[ \|f\|_{L^2(\R^d,\langle v \rangle^{(\gamma+2s)})} \leq \|f\|_{N^{s,\gamma}},\]
then for any $g \in L^2(\R^d,\langle v \rangle^{-(\gamma+2s)})$,
\[ \|g\|_{N^{-s,\gamma}} \leq \|g\|_{L^2(\R^d,\langle v \rangle^{-(\gamma+2s)})}.\]
\end{proof}

\begin{lemma} \label{l:cutting-large-velocities}
Consider a smooth function $b : \R^d \to [0,1]$ so that $b(v)=1$ for $v \in B_1$ and $b(v)=0$ for $v \notin B_2$. Let $b_R(v) = b(v/R)$. Let $D = (0,T) \times \Omega \times \R^d$. For any function $f \in \Hk(D)$, the product $b_R(v)f(t,x,v)$ converges to $f$ in $\Hk(D)$ as $R \to \infty$.
\end{lemma}

\begin{proof}
By definition, the norm $\|f\|_{\Hk(D)}$ consists of three terms
\begin{align*}
\|f\|_{\Hk(D)}^2 &= \iiint_{(0,T) \times \Omega \times \R^d} \langle v \rangle^{\gamma+2s} \abs{f(v)}^2 \ud v \ud x \ud t \\
&\phantom{=} + \iint_{(0,T) \times \Omega} \iint_{\df(v,v')<1} \langle v \rangle^{\frac{\gamma+2s+1}2} \langle v' \rangle^{\frac{\gamma+2s+1}2} \frac{|f(v')-f(v)|^2}{\df(v,v')^{d+2s}} \,\ud v' \ud v \ud x \ud t \\
&\phantom{=} + \iint_{(0,T) \times \Omega} \nm{\trp f}_{N^{-s,\gamma}}^2 \ud x \ud t.
\end{align*}
Applying the above formula to $\|f - b_R f\|_{\Hk}^2 = \|(1 - b_R) f\|_{\Hk}^2$, it is relatively easy to prove the convergence to zero of the first two terms using the dominated convergence theorem.

For the last term, we observe that $\trp (b_Rf) = b_R(v) \trp f$. Applying Lemma \ref{lem:negative-norm}, we know that there exists a function $F \in \Hk$ such that $\trp f = \mathcal{L}F$. Therefore, $\trp \left[(1-b_R)f\right] = (1-b_R) \mathcal{L}F$. We then use the commutator estimate of Lemma \ref{l:commutator} and have
\[ \nm{(1-b_R) \mathcal{L}F - \mathcal{L}\left[(1-b_R)F\right]}_{N^{-s,\gamma}} \leq M_{b_R} \|F\|_{N^{s,\gamma}}.\]
We observe that $M_{b_R} \to 0$ as $R \to \infty$. Indeed,
\begin{align*}
\left\vert \int_{\df(v,v')<1} \langle v \rangle^{\frac 12}\langle v' \rangle^{\frac 12} \frac{b_R(v')-b_R(v)}{\df(v,v')^{d+2s}} \,\ud v' \right\vert &\lesssim R^{-2},\\
 \left(\int_{\df(v,v')<1} \langle v \rangle^{\frac 12}\langle v' \rangle^{\frac 12} \frac{|b_R(v')-b_R(v)|^2}{\df(v,v')^{d+2s}} \,\ud v'\right)^{1/2} &\lesssim R^{-1}.
\end{align*}

In particular, $(1-b_R) \mathcal{L}F - \mathcal{L}\left[(1-b_R)F\right] \to 0$ in $N^{-s,\gamma}$. Finally, following the same analysis as above, we deduce that $(1-b_R)F \to 0$ in $N^{s,\gamma}$. Thus, $\mathcal{L}\left[(1-b_R)F\right] \to 0$ in $N^{-s,\gamma}$, which combined with the commutator estimate tells us that that $(1-b_R) \mathcal{L}F \to 0$ in $N^{-s,\gamma}$.
\end{proof}

\subsection{Galilean Group Operation and Convolution}

The Galilean group is the natural algebraic structure that is compatible with kinetic equations.
Let $\circ$ denote the Galilean group operation
\[(t_1,x_1,v_1) \circ (t_2,x_2,v_2) = (t_1+t_2, x_1+x_2+t_2 v_1, v_1+v_2).\]
This operation defines a non-commutative Lie group structure in $\R^{1+2d}$.
Note that the differential operators $\trp$, $\nabla_v$ and $\nabla_x$ are \emph{left}-invariant by the action of the group, and thus the Galilean translation is the correct group of transformations associated with this type of kinetic equations.


The convolution of functions is computed in terms of the Galilean group
\[ f \ast g (z) := \int_{\R^{1+2d}} f(w) g(w^{-1} \circ z) \,\ud w. \]
This convolution is associative, but it is not commutative. If we make the change of variables $w^{-1} \circ z \mapsto w$, we obtain the equivalent expression
\[ f \ast g (z) := \int_{\R^{1+2d}} f(z \circ w^{-1}) g(w) \,\ud w. \]
Note that $f \ast g$ is $C^\infty$ provided that at least one of the two functions is $C^\infty$.

Using convolutions in terms of the Galilean group with an appropriately scaled family of mollifiers provides a convenient way to approximate functions in $\Hk$ with smooth ones.
Let $\eta: \R^{1+2d} \to \R$ be a non-negative compactly-supported smooth function with integral one. 
We use kinetic scaling to produce a family of mollifiers:
\begin{align}\label{mollifier}
    \eta_\ep(t,x,v) = \ep^{-2s-(2+2s)d} \eta\left(\ep^{-2s} t, \ep^{-1-2s}x, \ep^{-1} v\right).
\end{align}
Given $f(t,x,v)$, we may consider the mollification
\[ f_{\ep}:=\eta_\ep \ast f. \]
This convolution commutes with the differential operators $\nabla_x$, $\nabla_v$, and $\trp$, as well as with any integro-differential operator in $v$, like $(-\Delta)_v^{s/2}$ or $(-\Delta)_v^{-s/2}$. To see this, observe that
\begin{align*}
\partial_{x_i} f(t,x,v) = \lim_{h \to 0} \frac{f((t,x,v)\circ(0,he_i,0)) - f(t,x,v)}h, \\
\partial_{v_i} f(t,x,v) = \lim_{h \to 0} \frac{f((t,x,v)\circ(0,0,he_i)) - f(t,x,v)}h, \\
\trp f(t,x,v) = \lim_{h \to 0} \frac{f((t,x,v)\circ(h,0,0)) - f(t,x,v)}h, \\
(-\Delta)_v^{s/2} f(t,x,v) = c_{s,d} \int_{\R^d} f((t,x,v) \circ (0,0,w)) |w|^{-d+s} \ud w.
\end{align*}
The point of these formulas is that they are all written in terms of translations of $f$ by the action of the Galilean group on the right, whereas the convolution $\eta \ast f$ is a weighted average of translations of $f$ by the action of the group on the left. Therefore, they trivially commute.

\begin{lemma} \label{l:density-of-smooth-in-Hk}
Let $\Omega$ be a bounded open set with a Lipschitz boundary and $D = (0,T)\times \Omega \times \R^d$. For any $\delta>0$, any function $f \in \Hk(D)$ can be approximated in $\Hk([\delta,T] \times \Omega \times \R^d)$ by smooth functions with compact support.
\end{lemma}

\begin{proof}
From Lemma \ref{l:cutting-large-velocities}, we see that functions with compact support are dense in $\Hk(D)$. It remains to prove that if $f \in \Hk(D)$ has compact support, then it can be approximated as the limit of a sequence of smooth functions with compact support.

The domain $\Omega$ has a Lipchitz boundary. Using a partition of unity corresponding to an appropriate covering of the boundary, we decompose the function $f$ as a sum of finitely many terms, each of them supported either away from $\partial \Omega$, or in some small ball where $\partial \Omega$ coincides with the graph of a Lipzhitz function.

The terms that are supported away from $\Omega$ are easily approximated by smooth functions using a standard convolution with respect to the Galilean group. The approximation of the terms whose support intersects $\partial \Omega$ requires some further explanation. For the rest of the proof, let us say that the function $f$ is one of these terms. Without loss of generality, let us assume that $f$ is supported in $t \in (0,T)$, $x \in \Omega \cap B_1$ and $v \in B_R$. Assume further that $\partial \Omega \cap B_2$ coincides the graph of some Lipchitz function. Since $f$ equals zero for $x \notin B_1$, it makes no difference at this point to assume that $\Omega$ is the supergraph of some global Lipschitz function.

The value of the convolution $[\eta_\ep \ast f](t,x,v)$ depends on the values of $f$ in certain neighborhood of $(t,x,v)$. When $x$ is very close to $\partial \Omega$, the value of $[\eta_\ep \ast f](t,x,v)$ will not be well determined as soon as this neighborhood contains points outside of $\Omega$. A natural workaround would be to first construct an extension operator from $\Hk((0,T)\times \Omega \times \R^d)$ to $\Hk((0,T)\times \R^d \times \R^d)$. However, such an extension is nontrivial. We follow a different idea by choosing special functions $\eta_\ep$ so that the value of $[\eta_\ep \ast f]$ depends on the values of $f$ in $D$ only.

We approximate $f$ by a convolution with a special function $\eta_\ep$ so that the values of $[\eta_\ep \ast f](t,x,v)$, for $t \in (\delta,T]$, $x \in B_1 \cap \Omega$, and $v \in \R^d$, depend only on the values of $f$ in $(0,T)\times \Omega \times \R^d$. Recall that
\[ [\eta_\ep \ast f](t,x,v) = \iiint \eta_\ep(s,y,w) f(t-s,x-y-w(t-s),v-w) \; \ud w \ud y \ud s.\]
We want that whenever $x \in \Omega \cap B_1$ and $t \in [\delta,T]$, the integrand is only nonzero for values of $s$, $y$ and $w$ so that $t-s \in (0,T)$ and $x-y-w(t-s) \in \Omega$.

We achieve our first goal by making $\eta_\ep$ supported in $s \in [0,\delta]$. For the second goal, we use that $\partial \Omega \cap B_2$ is the graph of a Lipschitz function. Because of the Lipschitz regularity of $\partial \Omega$, we know that $x - y - w(t-s) \in \Omega$ provided that $x \in \Omega$ and $y+w(t-s)$ belongs to certain cone of directions depending on $\partial \Omega$. Let $C_1 \subset \R^d$ be this cone. We want $\eta_\ep(s,y,w)$ to be supported in some subset of $[0,\delta] \times C_1 \times C_1$.

Let $\eta_1$ be some smooth function, with unit integral, whose support is inside $[0,\delta] \times C_1 \times C_1$. We set $\eta_\ep(s,y,w)$ in terms of $\eta_1$ as in \eqref{mollifier}. For any $\ep<1$, the support of $\eta_\ep$ is also contained in $[0,\delta]\times C_1 \times C_1$. Therefore, the values of $\eta_\ep \ast f$ in $[\delta,T] \times (\Omega \cap B_1) \times \R^d$ depend only on the values of $f$ in $(0,T) \times (\Omega \cap B_2) \times \R^d$.

At this point, we have a well defined one-parameter family of mollifications $f_\ep := \eta_\ep \ast f$. We must show that it converges to $f$ in $\Hk(D)$ as $\ep \to 0$. As explained initially, we can and do assume that $f$ is compactly supported.

Since $f$ is compactly supported in $|v| \leq R$, we also have that $f_\ep$ is supported in $|v| \leq R+\ep$. Combining the compact support of $f$ with Lemmas \ref{l:comparability-of-sobolev-norms} and \ref{l:comparability-of-negative-sobolev-norms}, we observe that it is enough to prove that
\begin{enumerate}
    \item[(a)]\[ f_\ep \to f \in L^2_{t,x}\left([\delta,T] \times (\Omega \cap B_1), H^s_v(\R^d)\right) \]
    \item[(b)]\[ \trp f_\ep \to \trp f \in L^2_{t,x}\left([\delta,T] \times (\Omega \cap B_1), H^{-s}_v(\R^d)\right).\]
\end{enumerate}

The first step to prove (a) is to verify that $f_\ep$ converges to $f$ in $L^2$. This is a standard consequence of the density of continuous functions in $L^2$.

In order to prove the convergence of $f_\ep$ to $f$ in $L^2_{t,x} H^s_v$ we use the well known fact that for any function $g : \R^d \to \R$,
\[ \|g\|_{H^s}^2 \approx \|g\|_{L^2}^2 + \|(-\Delta)^{s/2}g\|_{L^2}^2. \]
Moreover, $(-\Delta)_v^{s/2} [\eta_\ep \ast f] = \eta_\ep \ast \left( (-\Delta)_v^{s/2} f \right)$, which converges to $(-\Delta)_v^{s/2} f$ in $L^2$ as $\ep \to 0$.

To prove (b), we use the well known fact that for any function $g : \R^d \to \R$,
\[ \|g\|_{H^{-s}}^2 \approx \|(I-\Delta)^{-s/2}g\|_{L^2}^2. \]
Moreover, $(-\Delta)_v^{-s/2} \trp [\eta_\ep \ast f] = \eta_\ep \ast \left( (-\Delta)_v^{-s/2} \trp f \right)$, which converges to $(-\Delta)_v^{-s/2} \trp f$ in $L^2$ as $\ep \to 0$.
\end{proof}

\begin{remark}
It would be a bad idea to attempt a change variables to flatten the boundary of $\Omega$ for the proof of Lemma \ref{l:density-of-smooth-in-Hk}. Such a change of variables (see for example \cite{silvestre2022}) introduces an extra term to the transport operator, of the form $b(x,v) \cdot \nabla_v f$. In the case of $s=1$ (as in \cite{silvestre2022}), and for $C^{1,1}$ boundaries, this term is absorbed into the $L^2_{t,x} H^1_v$ norm (at least locally). In this paper we have $s < 1$ and there is no apparent way to control this extra term.
\end{remark}

\begin{remark}
Note that while $\eta \ast [(-\Delta)_v^{s/2} f] = (-\Delta)_v^{s/2} [\eta \ast f]$ and $\eta \ast [(-\Delta)_v^{-s/2} f] = (-\Delta)_v^{-s/2} [\eta \ast f]$, it is not true that $\mathcal L[\eta\ast f] = \eta \ast [\mathcal L f]$ for the operator $\mathcal L$ of Lemma \ref{lem:negative-norm}.
\end{remark}

\begin{lemma}[Smooth approximation] \label{lem:mollification}
Assume $f\in\Hk(D)$ satisfies the equation \eqref{eqn:boltzmann} in $D$ in the sense of distributions. There exists a sequence of smooth approximations $f_n$ so that $f_n \in C_c^\infty([1/n,T] \times \Omega \times \R^d)$ and
\begin{enumerate}
    \item[(i)]  $f_n$ converges to $f$ in $L^2_{t,x}N^{s,\gamma}_v(D)$ as $n \to \infty$.
    \item[(ii)] $\trp f_n$ converges to $\trp f$ in $L^2_{t,x}N^{-s,\gamma}_v(D)$.
    \item[(iii)] The function $f_n$ satisfies (classically) an equation in $D$ of the form
    \[ \trp f_\ep - Q_\ep = 0,\]
    where $Q_\ep \to \Qff$ in $L^2_{t,x}N^{-s,\gamma}_v(D)$ as $n \to \infty$.
\end{enumerate}
\end{lemma}

\begin{proof}
We construct the sequence $f_n$ using Lemma \ref{l:density-of-smooth-in-Hk}. The items (i) and (ii) are simply the convergence of $f_n$ to $f$ in $\Hk(D)$. For (iii), we set $Q_\ep := \trp f_\ep$ and use that $f$ satisfies the equation \eqref{eqn:boltzmann} in the sense of distributions.
\end{proof}

The following lemma provides a trace operator to a weighted $L^2$ space on the boundary $\BC$.
\begin{proposition} \label{p:trace}
The restriction operator $f \mapsto f|_{\BC}$ is well-defined from $\Hk(D)$ to $L^2_{loc}(\BC \cap \overline D,\omega)$ for the weight $\omega = \min\left\{|v \cdot n| , (v \cdot n)^2\right\}$. More precisely, for any $\widetilde \BC$ that is compactly contained in $\BC \cap \overline{D}$, there is a constant $C$ so that for all $f \in \Hk(D)$, we have
\[\int_{\widetilde \BC} f^2 \omega \,\ud \BC \leq C \|f\|^2_{\Hk(D)}.\]
Moreover, if $f_\ep \to f$ strongly in $L^2_{t,x} N_v^{s,\gamma}$ and $\trp f_\ep \to \trp f$ weakly in $L^2_{t,x} N_v^{-s,\gamma}$, then $f_\ep \to f$ strongly on $L^2_{loc}(\BC\cap \overline{D},\omega)$.
\end{proposition}

\begin{proof}
The result is similar to  \cite[Proposition 4.3]{silvestre2022}. Based on Lemma \ref{lem:mollification}, it suffices to consider the case of $f$ smooth. Denote $\varphi_+$ and $\varphi_-$ to be the Lipchitz functions whose values on $\BC$ are given by
\[\varphi_+:=\min\left(1,(v\cdot n(x))_+\right),\qquad \varphi_-:=\min\left(1,(v\cdot n(x))_-\right).\]
We assume that $\p\Omega$ is Lipschitz. Therefore, these functions are Lipschitz on $\BC$. We extend them to the interior of $\Omega$ in any way that preserves their sign and Lipchitz norm. Note that, $\varphi_+-\varphi_-\geq0$ and $\approx\omega$ on $\BC$. 

Let $\eta$ be a smooth compactly supported function that equals one on $\widetilde\BC$. Then we have
\begin{align*}
    \iiint_{\BC}\varphi_+\eta f^2(v\cdot n)&=\int_{D}\trp(\varphi_+\eta f^2)\\
    &=\int_{D}\trp(\varphi_+\eta) f^2+\int_{D}2\varphi_+\eta f\trp f\\
    &\ls\lnm{\trp(\varphi_+\eta)}\tnm{f}^2+\nm{2\varphi_+\eta f}_{L^2_{t,x} N_v^{s,\gamma}}\nm{\trp f}_{L^2_{t,x} N_v^{-s,\gamma}}\\
    &\ls\nm{f}_{\Hk(D)}^2.
\end{align*}
Hence, the desired bound follows on $\BC_+$ part. A similar argument justifies the $\BC_-$ part when considering $-\trp(\varphi_-\eta f^2)$.
\end{proof}

\begin{remark}
The proof of Proposition \ref{p:trace} is the only place in this paper where we use the assumption that the boundary of $\Omega$ is $C^{1,1}$. A Lipschitz boundary suffices for the rest of the analysis.
\end{remark}

%

In the next few lemmas, we show that the definitions given before for the boundary condition in terms of test functions imply a more classical condition in terms of the trace operator given in Proposition \ref{p:trace}.

\begin{lemma} \label{lem:influx-trace}
Let $f$ be a weak solution to the Boltzmann equation \eqref{eqn:boltzmann} with in-flow boundary condition $g$ in the sense of Definition \ref{d:weak-sol-inflow}. Then $g = f|_{\BC_-}$, where $f|_{\BC_-}$ is the trace of $f$ as in Proposition \ref{p:trace}.
\end{lemma}

\begin{proof}
When $f$ and $\varphi$ are both smooth functions, and $\varphi|_{\Gamma_+}=0$, integration by parts yields
\begin{align}\label{weak-BC-inflow=}
\iiint_{(0,T)\times\Omega\times\R^d} \left\{\trp f \,\varphi + f \,\trp \varphi \right\}\,\ud v\ud x\ud t=\iiint_{\BC_-} f \varphi \,(v \cdot \vn) \ud v\ud S_x\ud t.
\end{align}

According to Lemma \ref{lem:mollification}, we can approximate any function $f \in \Hk$ with a sequence of smooth functions $f_k \to f$ in $\Hk$, so that each $f_k$ satisfies the identity above.

From Proposition \ref{p:trace}, the boundary values $(f_k)|_{\BC_-}$ converge to $f|_{\BC_-}$ in $L^2_{loc}(\BC,\omega)$. Moreover, since $\varphi = 0$ on $\BC$, we must also have $|\varphi| \ls \omega$ on $\BC$. We can pass to the limit every term of the identity to get
\[ \iiint_{(0,T)\times\Omega\times\R^d} \left\{\trp f \,\varphi + f \,\trp \varphi \right\}\,\ud v\ud x\ud t=\iiint_{\BC_-} f|_{\BC_-} \varphi \,(v \cdot \vn) \ud v\ud S_x\ud t.
\]

When the function $f$ is a weak solution of the equation in the sense of Definition \ref{d:weak-sol-inflow}, we know in addition that $\trp f = \Qff$ in the sense of distributions. Replacing in the indentity above,
\[ \iiint_{(0,T)\times\Omega\times\R^d} \left\{ \Qff \,\varphi + f \,\trp \varphi \right\}\,\ud v\ud x\ud t=\iiint_{\BC_-} f|_{\BC_-} \varphi \,(v \cdot \vn) \ud v\ud S_x\ud t.
\]
This is the same as the identity in Definition \ref{d:weak-sol-inflow} but with $f|_{\BC_-}$ instead of $g$. Both identities hold for any test function $\varphi$, thus $f|_{\BC_-} = g$.
\end{proof}

The following two lemmas are proved by density in the same way as Lemma \ref{lem:influx-trace}.

\begin{lemma} \label{lem:influx-specular}
Let $f$ be a weak solution to the Boltzmann equation \eqref{eqn:boltzmann} with specular-reflection boundary condition in the sense of Definition \ref{d:weak-sol-specular}. Then $f|_{\BC}(t,x,v) = f|_{\BC}(t,x,\mathcal R_x v)$, where $f|_{\BC}$ is the trace of $f$ as in Proposition \ref{p:trace}.
\end{lemma}

\begin{lemma} \label{lem:influx-bounce-back}
Let $f$ be a weak solution to the Boltzmann equation \eqref{eqn:boltzmann} with bounce-back boundary condition in the sense of Definition \ref{d:weak-sol-bounce-back}. Then $f|_{\BC}(t,x,v) = f|_{\BC}(t,x,-v)$, where $f|_{\BC}$ is the trace of $f$ as in Proposition~\ref{p:trace}.
\end{lemma}



\section{Chain Rule for Weak Solutions} 

In order to establish inequalities for the truncated energy dissipation of the function $f$, we want to compute the equation satisfied by composed functions of the form $\psi(t,f) = \left(f - a(t)\right)_+^2$. Starting from the definitions we have given for the equation in the sense of distributions, for each type of boundary condition, we verify some form of the chain rule in this section.

In the following, let $\psi(t,y)$ be a generic function of $t$ and $y$. In the case of interest, it is Lipschitz with respect to $t$, differentiable with respect to $y$, and $\partial_y \psi$ is locally Lipschitz. We use $\psi_y$ to denote the derivative of $\psi$ with respect to $y$, and $\psi_t$ the derivative with respect to $t$. The objective of this section is to express the equation satisfied by $\psi(t,f)$ in the sense of distributions. While $\psi_y$ is locally Lipschitz, it is possible that it becomes unbounded for large values of $y$. Aiming at $\psi(t,f) = \left(f - a(t)\right)_+^2$, with $a(t)$ Lipschitz, we may assume that $\psi_{yy}$ and $\psi_{ty}$ are globally bounded. For technical reasons, it is convenient to start our analysis with functions $\psi$ so that $\psi_t$, $\psi_y$, $\psi_{ty}$ and $\psi_{yy}$ are all globally bounded. We will later approximate the case of interest by truncation.

\begin{lemma} \label{lem:weak-sol-common}
Let $f \in \Hk(D)$ be a weak solution of \eqref{eqn:boltzmann} in $D$. Let $g \in L^2(\BC,\omega)$ be the trace of $f$ in the sense of Proposition \ref{p:trace}. Assume that $\psi$ and $\psi_y$ are Lipschitz with $\psi_t$, $\psi_y$, $\psi_{ty}$ and $\psi_{yy}$ globally bounded.
Then, for any test function $\varphi \in C^1_c\left((0,T)\times \overline \Omega \times \R^d\right)$, so that $\varphi=0$ on $\BC_0$, we have
\begin{align}\label{eq:chain-common}
\iiint_{(0,T) \times \Omega\times\R^d} \left\{\psi(t,f) \, \trp \varphi + \psi_y(t,f)\Qff \varphi+\psi_t(t,f)\varphi \right\}\,\ud v\ud x\ud t \\
=\iiint_{\BC} \psi(t,g) \varphi \,(v \cdot \vn) \, \ud v\ud S_x\ud t. \no
\end{align}
\end{lemma}

\begin{proof}
We apply Lemma \ref{lem:mollification}. We have that the smooth approximate function $f_\ep$ satisfies (classically)
\[ \trp f_\ep - Q_\ep = 0 \quad\text{in } D.\]
Multiplying the equation with $\psi_y(t,f_\ep)$, we get
\[ \trp \psi(t,f_\ep) - \psi_y(t,f_\ep) Q_\ep - \psi_t(t,f_\ep) = 0 \quad\text{in } D.\]

Multiplying the equation by $\varphi$ and integrating the transport term by parts we get

\begin{equation} \label{eq:eq-for-fep}
\iiint_D \psi(t,f_\ep) \trp \varphi + \psi_y(t,f_\ep) Q_\ep \varphi + \psi_t(t,f_\ep) \varphi \, \ud v \ud x \ud t \\
= \iiint_{\BC} \psi(t,f_\ep) (v\cdot n) \varphi \, \ud v \ud S_x \ud t.
\end{equation}

Let us assume initially that $\psi(t,y)$ is globally Lipschitz with respect to $y$. We will get rid of this assumption later on. Moreover, we recall that $\psi \in C^{1,1}$. Its second derivatives are bounded.

Since $\psi$ is globally Lipschitz, and $f_\ep \to f$ in $L^2(D)$ (and in particular in $L^1_{loc}$), we see that the first term converges for any fixed test function $\varphi$.
\[ \iiint \psi(t,f_\ep) \trp \varphi \, \ud v \ud x \ud t \:\to \iiint \psi(t,f) \trp \varphi \, \ud v \ud x \ud t \quad\text{as } \ep \to 0.\]

Since $\psi \in C^{1,1}$, we observe that $\psi_y$ is globally Lipschitz. Since $f_\ep \to f$ in $L^2_{t,x} N^{s,\gamma}_v$ and $\psi_y$ is Lipschitz in $y$, we see that $\psi(t,f_\ep)$ is bounded uniformly in $L^2_{t,x} N^{s,\gamma}_v$. Since it clearly converges in $L^2_{loc}$, then it must converge at least weakly in $L^2_{t,x} N^{s,\gamma}_v$.

Since $Q_\ep \to \Qff$ in $L^2_{x,v} N^{-s,\gamma}_v$, then we deduce the convergence of the second term
\[ \iiint \psi_y(t,f_\ep) Q_\ep \varphi \, \ud v \ud x \ud t \to \iiint \psi_y(t,f) \Qff \varphi \, \ud v \ud x \ud t\]


For the next term, we use that $\psi_{ty}$ is bounded, so that $\psi_t(t,f_\ep) \to \psi_t(t,f)$ in $L^2_{loc}$ and
\[ \iiint \psi_t(t,f_\ep) \varphi \, \ud v \ud x \ud t \:\to \iiint \psi_t(t,f) \varphi \, \ud v \ud x \ud t.\]

Since $\varphi \in C^1$ and $\varphi=0$ on $\BC_0$, we have that $|\varphi(t,x,v)| \lesssim d((x,v), \BC_0)$. In particular, $|\varphi(t,x,v) (v\cdot n)| \lesssim \omega$ on $\BC$. For the boundary term, we use that $(f_\ep)|_{\BC} \to g$ in $L^2(\BC,\omega)$. Since $\psi$ is globally Lipschitz, we deduce that $\psi(t,f_\ep) \to \psi(t,g)$ in $L^2(\BC,\omega)$. Therefore
\[ \iiint_\BC \psi(t,f_\ep) (v\cdot n) \varphi \, \ud v \ud S_x \ud t \to \iiint_\BC \psi(t,g) (v\cdot n) \varphi \, \ud v \ud S_x \ud t.\]

We pass to the limit every term in \eqref{eq:eq-for-fep} and finish the proof.
\end{proof}

Lemma \ref{lem:weak-sol-common} already suffices to establish an equality for weak solutions with the inflow boundary conditions. In this case, we normally use test functions $\varphi$ that vanish on $\BC_+$. In particular, they vanish on $\BC_0$ as well.

\begin{corollary}[Restricted chain rule for weak solutions with in-flow/diffusive-reflection boundary] \label{cor:chain-weak-sol-inflow}
Let $f$ be a weak solution of \eqref{eqn:boltzmann} in $D$ satisfying the in-flow boundary condition. Assume that $\psi$ and $\psi_y$ are Lipschitz with $\psi_t$, $\psi_y$, $\psi_{ty}$ and $\psi_{yy}$ globally bounded. Then for any test function $\varphi \in C^1_c\left((0,T)\times\overline\Omega\times\R^d\right)$ so that $\varphi|_{\BC_+}=0$, we have
\begin{align}\label{eq:chain-weak-sol-inflow}
\iiint_{(0,T) \times \Omega\times\R^d} &\left\{\psi(t,f) \, \trp \varphi + \psi_y(t,f) \Qff \varphi+\psi_t(t,f)\varphi \right\} \,\ud v\ud x\ud t \\
&=\iiint_{\BC_-} \psi(t,f|_{\BC_-}) \varphi \, (v \cdot \vn) \ud v \ud S_x \ud t. \no
\end{align}
\end{corollary}

\begin{proof}
The test function $\varphi$ vanishes on $\BC_+$ and is continuous. Therefore, it vanishes on $\BC_0$ and Lemma~\ref{lem:weak-sol-common} applies directly.
\end{proof}

\begin{lemma} \label{lem:normofeta}
Let $\dist$ represent the Euclidean distance in $\R^{1+2d}$. Let $\chi:[0,\infty) \to [0,1]$ be a smooth function so that $\chi(w) = 1$ if $w \in [0,1]$ and $\chi(w)=0$ if $w \geq 2$.
For any large radius $R>0$ and $\ep \in (0,1)$, denote
\begin{align} \label{eq:eta}
   \eta_{\ep}(z):=\left(1 - \ep^{-1}\dist(z,\BC_0)\right)_+ \chi(R^{-1}|v|).
\end{align}
Then, we have
\[ \|\eta_\ep(t,x,\cdot) \|_{N^{s,\gamma}} \leq \begin{cases}
0 & \text{if } \dist(x,\partial \Omega) \geq \ep, \\
C(R) \ep^{1/2-s} &\text{otherwise}.
\end{cases}
\]
Here $C(R)$ is a constant depending on $d$, $\gamma$, $s$ and $R$.
\end{lemma}

\begin{proof}
If $\dist(x,\partial \Omega) \geq \ep$, then we will always have $\dist((t,x,v),\BC_0) > \ep$ for any value of $t$ and $v$. Therefore $\eta(t,x,\cdot)\equiv 0$ in that case.

If $\dist(x,\partial \Omega) < \ep$, we analyze the formula for $\|\eta_\ep(t,x,\cdot) \|_{N^{s,\gamma}}$. Our estimate is based on the following four simple facts about $\eta_\ep$.
\begin{itemize}
    \item[(i)] $\eta_\ep(t,x,v) \leq 1$ for every value of $(t,x,v)$.
    \item[(ii)] $\eta_\ep(t,x,v) \neq 0$ only if $\dist((t,x,v),\BC_0) \leq \ep$.
    \item[(iii)] $|\nabla_v \eta_\ep| \leq \ep^{-1}$.
    \item[(iv)] $\eta_\ep(t,x,v) = 0$ if $|v| > R$.
\end{itemize}

According to \eqref{eq:normNsg}, for every fixed value of $t$ and $x$ (which we omit to avoid clutter) we have
\begin{align*}
 \nm{\eta_\ep}_{N^{s,\gamma}}^2 &=\int_{\R^d}\br{v}^{\gamma+2s}\abs{\eta_\ep(v)}^2\ud v+\iint_{\R^d\times\R^d}\br{v}^{\frac{\gamma+2s+1}{2}}\br{v'}^{\frac{\gamma+2s+1}{2}}\frac{\abs{\eta_\ep(v)-\eta_\ep(v')}^2}{\df(v,v')^{d+2s}}\mathds{1}_{\df(v,v')\leq 1}\,\ud v'\ud v.
\end{align*}

We estimate the first term using (i), (ii) and (iv).
\begin{align*}
\int_{\R^d}\br{v}^{\gamma+2s}\abs{\eta_\ep(v)}^2\ud v \ls R^{d-1+\gamma+2s} \ep.
\end{align*}

In order to estimate the second term, we observe that the integrand is nonzero only if $\dist((t,x,v),\BC_0)<\ep$ or $\dist((t,x,v'),\BC_0)<\ep$. By symmetry, it suffices to consider the set where $\dist((t,x,v),\BC_0)<\ep$ paired with any other value of $v'$:
\begin{align*}
 \iint_{\R^d\times\R^d} &\br{v}^{\frac{\gamma+2s+1}{2}}\br{v'}^{\frac{\gamma+2s+1}{2}} \frac{\abs{\eta_\ep(v)-\eta_\ep(v')}^2}{\df(v,v')^{d+2s}}\mathds{1}_{\df(v,v')\leq 1}\,\ud v'\ud v \\
 &\leq 2 R^{\gamma+2s+1} \int_{\{v: \dist((t,x,v),\BC_0)<\ep\} \cap B_R} \int_{\{v' \in \R^d : \df(v,v')<1\}} \frac{\abs{\eta_\ep(v)-\eta_\ep(v')}^2}{\df(v,v')^{d+2s}} \,\ud v'\ud v
 \intertext{Observe that $\df(v,v')$ is comparable to the usual Euclidean distance in any ball $B_{R+1}$, with factors depending on $R$.}
 &\ls C(R) \int_{\{v: \dist((t,x,v),\BC_0)<\ep\} \cap B_R} \int_{\{v' \in B_1(v)\}} \frac{\abs{\eta_\ep(v)-\eta_\ep(v')}^2}{|v-v'|^{d+2s}} \,\ud v'\ud v.
 \end{align*}
We split the domain of integration in the second integral between $v' \in B_{c(R)\ep}(v)$ and $v' \notin B_{c(R)\ep}(v)$. In the latter, we necessarily have $\eta_\ep(v') = 0$ (because of (ii) above). Therefore, for any $v$ such that $\dist((t,x,v),\BC_0)<\ep$, we have (using (i) above)
\[ \int_{\{v' \notin B_{c(R)\ep}(v)\}} \frac{\abs{\eta_\ep(v)-\eta_\ep(v')}^2}{|v-v'|^{d+2s}} \,\ud v' \ls \ep^{-2s}.\]
For the other term, we use (iii) above to get
\[ \int_{\{v' \in B_{c(R)\ep}(v)\}} \frac{\abs{\eta_\ep(v)-\eta_\ep(v')}^2}{|v-v'|^{d+2s}} \,\ud v' \leq \int_{\{v' \in B_{c(R)\ep}(v)\}} \frac{\ep^2 |v-v'|^2}{|v-v'|^{d+2s}} \,\ud v' \ls \ep^{-2s}.\]
Integrating over all $v \in B_R$ such that $\dist((t,x,v),\BC_0)<\ep$, we get
\[ \iint_{\R^d\times\R^d} \br{v}^{\frac{\gamma+2s+1}{2}}\br{v'}^{\frac{\gamma+2s+1}{2}} \frac{\abs{\eta_\ep(v)-\eta_\ep(v')}^2}{\df(v,v')^{d+2s}}\mathds{1}_{\df(v,v')\leq 1}\,\ud v'\ud v \ls \ep^{1-2s}. \]
\end{proof}

\begin{lemma} \label{lem:Nsgofcomposition}
Let $f$ be a function in $N^{s,\gamma}$ and $\Phi : \R \to \R$ be Lipschitz so that $\Phi(0)=0$. Then, $\Phi \circ f \in N^{s,\gamma}$ and we have
\[ \|\Phi \circ f\|_{N^{s,\gamma}} \lesssim \|f\|_{N^{s,\gamma}} \|\Phi\|_{Lip}.\]
\end{lemma}

\begin{proof}
We have
\begin{align*} 
\nm{\Phi \circ f}_{N^{s,\gamma}}^2:=&\int_{\R^d}\br{v}^{\gamma+2s}\abs{\Phi \circ f(v)}^2\ud v\\
&+\iint_{\R^d\times\R^d}\br{v}^{\frac{\gamma+2s+1}{2}}\br{v'}^{\frac{\gamma+2s+1}{2}}\frac{\abs{\Phi \circ f(v)-\Phi \circ f(v')}^2}{\df(v,v')^{d+2s}}\mathds{1}_{\df(v,v')\leq 1}\,\ud v'\ud v
\end{align*}
Notice that
\begin{align*}
    \abs{\Phi \circ f(v)}\leq \nm{\Phi}_{Lip}\abs{f(v)},\qquad 
    \abs{\Phi \circ f(v)-\Phi \circ f(v')}\leq \nm{\Phi}_{Lip}\abs{ f(v)- f(v')}.
\end{align*}
Hence, our result naturally follows.
\end{proof}

\begin{lemma} \label{lem:Nsgofproduct}
Let $f$ and $g$ be bounded and in $N^{s,\gamma}$. Then, their product also belongs to $N^{s,\gamma}$ and we have
\[ \|fg\|_{N^{s,\gamma}} \lesssim \|f\|_{N^{s,\gamma}} \|g\|_{L^\infty} + \|f\|_{L^\infty} \|g\|_{N^{s,\gamma}}.\]
\end{lemma}

\begin{proof}
We compute using the formula \eqref{eq:normNsg} for the norm in $N^{s,\gamma}$.

\[\begin{aligned}
\nm{fg}_{N^{s,\gamma}}^2 :=\int_{\R^d} & \br{v}^{\gamma+2s}\abs{f(v)g(v)}^2\ud v \\ & +\iint_{\R^d\times\R^d}\br{v}^{\frac{\gamma+2s+1}{2}}\br{v'}^{\frac{\gamma+2s+1}{2}}\frac{\abs{f(v)g(v)-f(v')g(v')}^2}{\df(v,v')^{d+2s}}\mathds{1}_{\df(v,v')\leq 1}\,\ud v'\ud v.\end{aligned} \]
For the first term, we clearly have
\[ \int_{\R^d}\br{v}^{\gamma+2s} \abs{f(v)g(v)}^2\ud v \leq \|g\|_{L^\infty}^2 \int_{\R^d}\br{v}^{\gamma+2s}\abs{f(v)}^2\ud v \leq \|g\|_{L^\infty}^2 \|f\|_{N^{s,\gamma}}^2. \]
The functions $f$ and $g$ in the right hand side are exchangeable in this case.

For the second term, we add and subtract a term in $|f(v)g(v)-f(v')g(v')| = |f(v)g(v)-f(v)g(v')+f(v)g(v')-f(v')g(v')| \leq |f(v)|\cdot |g(v)-g(v')|+|f(v)-f(v')| \cdot |g(v')|$. Thus, we get
\begin{align*}
\nm{fg}_{N^{s,\gamma}}^2 & \leq \|g\|_{L^\infty}^2 \|f\|_{N^{s,\gamma}}^2 \\ 
&\phantom{\leq} +2\iint_{\R^d\times\R^d}\br{v}^{\frac{\gamma+2s+1}{2}}\br{v'}^{\frac{\gamma+2s+1}{2}}\frac{|f(v)|^2 \abs{g(v)-g(v')}^2}{d(v,v')^{d+2s}}\mathds{1}_{d(v,v')\leq 1}\,\ud v'\ud v \\ 
&\phantom{\leq} +2\iint_{\R^d\times\R^d}\br{v}^{\frac{\gamma+2s+1}{2}}\br{v'}^{\frac{\gamma+2s+1}{2}}\frac{|g(v')|^2 \abs{f(v)-f(v')}^2}{\df(v,v')^{d+2s}}\mathds{1}_{\df(v,v')\leq 1}\,\ud v'\ud v \\
&\lesssim \|g\|_{L^\infty}^2 \|f\|_{N^{s,\gamma}}^2 + \|f\|_{L^\infty}^2 \|g\|_{N^{s,\gamma}}^2
\end{align*}
\end{proof}

\begin{lemma}[Unrestricted chain rule] \label{lem:chain-any-boundary}
Let $f$ be a weak solution of \eqref{eqn:boltzmann} in $D$ with trace $g \in L^2(\BC,\omega)$ (as in Proposition \ref{p:trace}). Assume that $\psi$ and $\psi_y$ are Lipschitz with $\psi_t$, $\psi_y$, $\psi_{ty}$ and $\psi_{yy}$ globally bounded. Assume further that $\psi_y(t,0)=0$ for all $t\in(0,T)$. Then, for any test function $\varphi \in C^1_c\left((0,T)\times\overline{\Omega}\times\R^{d}\right)$ (not necessarily vanishing on $\BC_+$),
we have
\begin{equation} \label{eq:chain-up-to-trace}
\begin{aligned}
\iiint_{(0,T)\times \Omega\times\R^d} &\left\{\psi(t,f) \,\trp \varphi + \psi_y(t,f)\Qff \varphi+\psi_t(t,f)\varphi \right\}\,\ud v\ud x\ud t \\
& = \lim_{\ep \to 0} \iiint_{\BC} \psi(t,g) \varphi \left(1-\eta_\ep(t,x,v)\right) \,(v \cdot \vn) \ud v\ud S_x\ud t.
\end{aligned}
\end{equation}
Here, $\eta_\ep$ is defined in Lemma \ref{lem:normofeta} with $R$ sufficiently large so that $\varphi(t,x,v) \neq 0$ only for $|v| < R$.
\end{lemma}

Note that the limit on the right hand side could be interpreted as a principal value of the integral
\[ \lim_{\ep \to 0} \iiint_{\BC} \psi(t,g) \varphi \left(1-\eta_\ep(t,x,v)\right) \,(v \cdot \vn) \ud v\ud S_x\ud t
=: PV  \iiint_{\BC} \psi(t,g) \varphi \,(v \cdot \vn) \ud v\ud S_x\ud t. \]

\begin{proof}
According to Lemma \ref{lem:weak-sol-common}, the equality holds whenever $\varphi$ vanishes on $\BC_0$ and $\psi$ has bounded first and second derivatives.

Let us start by considering the case that $\varphi$ does not necessarily vanishes on any part of $\BC$, but $\psi$ still has bounded first and second derivatives. We assume moreover that $\psi(t,0)=0$ and $\psi_y(t,0)=0$.

We follow the idea of \cite[Lemma 4.6]{silvestre2022}. Let $\varphi_{\ep}:=(1-\eta_{\ep})\varphi$, and thus $\varphi_{\ep}|_{\BC_0}=0$. Then we may apply Lemma \ref{lem:weak-sol-common} to obtain
\begin{align}\label{temp 01}
\iiint_{(0,T)\times\Omega\times\R^d} &\left\{\psi(t,f) \,\trp \varphi_{\ep} + \psi_y(t,f)\Qff \varphi_{\ep}+\psi_t(t,f)\varphi_{\ep} \right\}\,\ud v\ud x\ud t \\
=&\iiint_{\BC_-} \psi(t,g) \varphi_{\ep} \,(v \cdot \vn) \ud v\ud S_x\ud t. \no
\end{align}
Taking $\ep\rt0$, we need to consider the limit of each term. We do not do anything to the boundary term on $\BC$, since the equality in this Lemma involves the limit as $\ep \to 0$ explicitly.

Since $0\leq\varphi_{\ep}\leq\varphi$, due to dominated convergence theorem, we have
\begin{align*}
    \lim_{\ep\rt0}\iiint_{(0,T)\times\Omega\times\R^d}\psi_t(t,f)\varphi_{\ep}\,\ud v\ud x\ud t =&\iiint_{(0,T)\times\Omega\times\R^d}\psi_t(t,f)\varphi\,\ud v\ud x\ud t.
\end{align*}

Since $\psi_{yy}$ is bounded and $\psi_y(t,0)=0$, we apply Lemma \ref{lem:Nsgofcomposition} to deduce that $\nm{\psi_y(t,f)}_{L^2_{t,x}((0,t)\times\Omega,N^{s,\gamma})} \leq \nm{f}_{L^2_{t,x}((0,t)\times\Omega,N^{s,\gamma})} \sup |\psi_{yy}|$. Furthermore, since we assume that $\psi_y$ is bounded, we apply Lemma \ref{lem:Nsgofproduct} to get

\[
    \nm{\psi_y(t,f)\varphi_{\ep}}_{L^2_{t,x}((0,t)\times\Omega,N^{s,\gamma})}\ls \nm{f}_{L^2_{t,x}((0,t)\times\Omega,N^{s,\gamma})},
\]

We claim that $\psi_y(t,f)\varphi_{\ep}$ converges to $\psi_y(t,f)\varphi$ in $L^2_{t,x}((0,T)\times\Omega,N^{s,\gamma})$. Indeed,
\begin{align*} 
\iint_{(0,T) \times \Omega} &\nm{\psi_y(t,f)\varphi_{\ep} - \psi_y(t,f)\varphi}_{N^{s,\gamma}}^2 \ud x \ud t = \iint_{(0,T) \times \Omega} \nm{\psi_y(t,f) \eta_\ep(t,x,v) \varphi}_{N^{s,\gamma}}^2 \ud x \ud t \\
\intertext{Using Lemma \ref{lem:Nsgofproduct}}
&\lesssim \iint_{(0,T) \times \{ x \in \Omega: \dist(x,\partial \Omega)<\ep\}} \nm{\psi_y(t,f) \varphi}_{N^{s,\gamma}}^2 + (\sup |\psi_y \varphi|) \|\eta_\ep\|_{N^{s,\gamma}}^2
\end{align*}
The first term on the right hand side converges to zero as $\ep \to 0$ because $\psi_y(t,f)  \varphi \in L^2_{t,x} N^{s,\gamma}_v$. The second term converges to zero due to Lemma \ref{lem:normofeta}.

Since $\Qff \in L^2_{t,x} N^{-s,\gamma}$, we deduce that
\[\lim_{\ep \to 0} \iiint_{(0,T)\times\Omega\times\R^d} \psi_y(t,f) \Qff \varphi_{\ep} \,\ud v\ud x\ud t = \iiint_{(0,T)\times\Omega\times\R^d} \psi_y(t,f) \Qff \varphi \,\ud v \ud x \ud t. \]

Finally, notice that
\begin{align}\label{temp 02}
    \trp \varphi_{\ep}={(1-\eta_{\ep})}\trp\varphi- \varphi \trp{\eta_{\ep}}.
\end{align}
For the first term in \eqref{temp 02}, using the dominated convergence theorem, we have
\begin{align*}
 \lim_{\ep\rt0}\iiint_{(0,T)\times\Omega\times\R^d}\psi(t,f)(1-\eta_{\ep}) \trp\varphi\,\ud v\ud x\ud t=&\iiint_{(0,T)\times \Omega\times\R^d}\psi(t,f)\trp\varphi\,\ud v\ud x\ud t.
\end{align*}
For the second term in \eqref{temp 02}, we observe that, by construction, $\partial_t \eta_{\ep}=0$ everywhere, and $v\cdot\nabla_x \eta_{\ep} = 0$ except in an $\ep$-neighborhood of $\BC_0$ where
\[\abs{\trp\eta_\ep}\ls\ep^{-1}.\]
Hence, for $\psi(t,f)\in L^2_{loc}$, we have
\begin{align*}
    \iiint_{(0,T)\times\Omega\times\R^d}\psi(t,f)\trp\eta_{\ep} \varphi\,\ud v\ud x\ud t
    = \iiint_{\text{dist}(z,\BC_0)<\ep}\psi(t,f)\left(v \cdot \nabla_x\eta_{\ep} \right)\varphi\,\ud v\ud x\ud t,
\end{align*}
and
\begin{align*}
    &\abs{\iiint_{\text{dist}(z,\BC_0)<\ep}\psi(t,f)\left(v \cdot \nabla_x\eta_\ep\right)\varphi\,\ud v\ud x\ud t}\ls\ep^{-1}\iiint_{\text{dist}(z,\BC_0)<\ep}\abs{\psi(t,f)\varphi}\,\ud v\ud x\ud t\\
    \ls&\;\ep^{-1}\left(\iiint_{\text{dist}(z,\BC_0)<\ep}\abs{\psi(t,f)\varphi}^2\ud v\ud x\ud t\right)^{\frac{1}{2}}\cdot \abs{\left\{z\in\R^{1+2d}:\text{dist}(z,\BC_0)<\ep\right\}}^{\frac{1}{2}}\\
    \ls&\left(\iiint_{\text{dist}(z,\BC_0)<\ep}\abs{\psi(t,f)\varphi}^2\ud v\ud x\ud t\right)^{\frac{1}{2}}\rt 0 \quad\text{as } \ep\rt0.
\end{align*}

We have that $\psi(t,f)\in L^2_{loc}$ because we are assuming that $\psi$ is globally Lipschitz. Then we may pass to limit and obtain
\[\lim_{\ep\rt0}\iiint_{(0,T)\times\Omega\times\R^d}\psi(t,f)\trp\eta_\ep\varphi\,\ud v\ud x\ud t = 0.\]
Hence, we know that as $\ep\rt0$, \eqref{temp 01} converges to \eqref{eq:chain-up-to-trace}.

This establishes \eqref{eq:chain-up-to-trace} and finishes the proof of the lemma.
\end{proof}

It is important to note that to make sense of the left hand side in \eqref{eq:chain-up-to-trace} we need $\psi_{yy}$ and $\psi_{ty}$ to be bounded, but we do not need any global bound for $\psi_t$ and $\psi_y$. Indeed, for a function of the form $\psi(t,f) = \left(f-a(t)\right)_+^2$, with $a$ Lipschitz, every term on the left hand side of \eqref{eq:chain-up-to-trace} makes sense. By truncation, we could approximate any generic function so that $\psi_{yy}$ and $\psi_{ty}$ are bounded, with a sequence of functions $\psi_R$ with bounded first derivatives and second derivatives that coincide with $\psi$ whenever $|f|<R$. However, it is not immediately clear what the limit of the boundary integral of the right hand side would be in general. This difficulty is resolved by simplifying the boundary integral in each case of our three possible boundary conditions. We explore them one by one in the next three lemmas.

\begin{lemma}[Unrestricted chain rule for weak solutions with in-flow boundary] \label{lem:chain-sub--sol-inflow}
Let $f$ be a weak solution of \eqref{eqn:boltzmann} in $D$ satisfying the in-flow boundary condition. Assume $\psi$ and $\psi_y$ are locally Lipschitz, with $\psi_{yy}$ and $\psi_{ty}$ globally bounded. Assume further that $\psi(t,f)\geq0$ and that the trace of $f$ on $\BC_-$ is bounded. Then, for any non-negative test function $\varphi \in C^1_c\left((0,T)\times\overline \Omega \times \R^d\right)$ (not necessarily vanishing on $\BC_+$),
we have
\begin{align}\label{ieq:chain-sub-sol}
\iiint_{(0,T)\times\Omega\times\R^d} &\left\{\psi(t,f) \,\trp \varphi + \psi_y(t,f)\Qff \varphi+\psi_t(t,f)\varphi \right\}\,\ud v\ud x\ud t \\
& \geq \iiint_{\BC_-} \psi(t,f|_{\BC_-}) \varphi \,(v \cdot \vn) \ud v\ud S_x\ud t. \no
\end{align}
\end{lemma}

\begin{proof}
According to Lemma \ref{lem:chain-any-boundary}, if $\psi$ and $\psi_y$ are globally Lipschitz, we have
\[\begin{aligned}
\iiint_{(0,T)\times \Omega\times\R^d} &\left\{\psi(t,f) \,\trp \varphi + \psi_y(t,f)\Qff \varphi+\psi_t(t,f)\varphi \right\}\,\ud v\ud x\ud t \\
& = \lim_{\ep \to 0} \iiint_{\BC} \psi(t,f|_{\BC}) \varphi \left(1-\eta_\ep(t,x,v)\right) \,(v \cdot \vn) \ud v\ud S_x\ud t.
\end{aligned}
\]
We split the boundary term between the integral on $\BC_+$ and $\BC_-$. Since $\psi \geq 0$, the part of the integral on $\BC_+$ is nonnegative. Therefore
\begin{align*}
\iiint_{(0,T)\times \Omega\times\R^d} &\left\{\psi(t,f) \,\trp \varphi + \psi_y(t,f)\Qff \varphi+\psi_t(t,f)\varphi \right\}\,\ud v\ud x\ud t \\
& \geq \lim_{\ep \to 0} \iiint_{\BC_-} \psi(t,f|_{\BC_-}) \varphi \left(1-\eta_\ep(t,x,v)\right) \,(v \cdot \vn) \ud v\ud S_x\ud t, \\
\intertext{and since $f|_{\BC_-}$ is bounded we use the dominated convergence theorem to obtain}
& = \iiint_{\BC_-} \psi(t,f|_{\BC_-}) \varphi\,(v \cdot \vn) \ud v\ud S_x\ud t.
\end{align*}

If the derivatives of $\psi$ are not globally bounded, we construct a function $\psi_R$ that coincides with $\psi$ whenever $|f|<R$. We choose this function $\psi_R$ with bounded first and second derivatives, so that the inequality above applies. There is no difficulty in passing to the limit as $R \to \infty$ at this point.
\end{proof}

\begin{lemma}[Chain rule for weak solutions with bounce-back boundary] \label{lem:chain-weak-sol-bounce-back} 
Assume that $f$ is a weak solution of \eqref{eqn:boltzmann} in $D$ satisfying the bounce-back boundary condition. Let $\psi \in C^{1,1}$, with $\psi_{yy}$ and $\psi_{ty}$ bounded.
Then for any test function $\varphi \in C^1_c\left((0,T) \times \overline\Omega \times \R^d\right)$ so that $\varphi(t,x,v)=\varphi(t,x,-v)$ for all $x\in\partial\Omega$, we have
\[ \begin{aligned}
&\iiint_{(0,T)\times \Omega\times\R^d} \left\{\psi(t,f) \,\trp \varphi + \psi_y(t,f)\Qff \varphi+\psi_t(t,f)\varphi \right\}\,\ud v\ud x\ud t=0.
\end{aligned} \]
\end{lemma}

\begin{proof}
According to Lemma \ref{lem:chain-any-boundary}, if $\psi$ and $\psi_y$ are globally Lipschitz, we have
\[\begin{aligned}
\iiint_{(0,T)\times \Omega\times\R^d} &\left\{\psi(t,f) \,\trp \varphi + \psi_y(t,f)\Qff \varphi+\psi_t(t,f)\varphi \right\}\,\ud v\ud x\ud t \\
& = \lim_{\ep \to 0} \int_{\BC} \psi(t,f|_{\BC}) \varphi \left(1-\eta_\ep(t,x,v)\right) \,(v \cdot \vn) \ud v\ud S_x\ud t.
\end{aligned}
\]
Since $\varphi(t,x,v)=\varphi(t,x,-v)$, we observe that the boundary term vanishes for any value of $\ep>0$. Thus
\[ \iiint_{(0,T)\times \Omega\times\R^d} \left\{\psi(t,f) \,\trp \varphi + \psi_y(t,f)\Qff \varphi+\psi_t(t,f)\varphi \right\}\,\ud v\ud x\ud t  =0.
\]
Since the boundary term disappeared, there is now no difficulty in truncating a function $\psi$ with unbounded first derivatives and passing to the limit.
\end{proof}

\begin{lemma}[Chain rule for weak solutions with specular-reflection boundary] \label{lem:chain-weak-sol-specular} 
Assume that $f$ is a weak solution of \eqref{eqn:boltzmann} in $D$ satisfying the specular-reflection boundary condition. Let $\psi \in C^{1,1}$, with $\psi_{yy}$ and $\psi_{ty}$ bounded.
Then for any test function $\varphi \in C^1_c\left((0,T) \times \overline\Omega \times \R^d\right)$ so that $\varphi(t,x,v)=\varphi(t,x,\mathcal{R}_xv)$ for all $x\in\partial\Omega$, we have
\[ \begin{aligned}
&\iiint_{(0,T)\times \Omega\times\R^d} \left\{\psi(t,f) \,\trp \varphi + \psi_y(t,f)\Qff \varphi+\psi_t(t,f)\varphi \right\}\,\ud v\ud x\ud t=0.
\end{aligned} \]
\end{lemma}

\begin{proof}
According to Lemma \ref{lem:chain-any-boundary}, if $\psi$ has bounded first and second derivatives we have
\[\begin{aligned}
\iiint_{(0,T)\times \Omega\times\R^d} &\left\{\psi(t,f) \,\trp \varphi + \psi_y(t,f)\Qff \varphi+\psi_t(t,f)\varphi \right\}\,\ud v\ud x\ud t \\
& = \lim_{\ep \to 0} \int_{\BC} \psi(t,f|_{\BC}) \varphi \left(1-\eta_\ep(t,x,v)\right) \,(v \cdot \vn) \ud v\ud S_x\ud t.
\end{aligned}
\]
Since $\varphi(t,x,v)=\varphi(t,x,\mathcal{R}_xv)$, we observe that the boundary term vanishes for any value of $\ep>0$. Thus
\[ \iiint_{(0,T)\times \Omega\times\R^d} \left\{\psi(t,f) \,\trp \varphi + \psi_y(t,f)\Qff \varphi+\psi_t(t,f)\varphi \right\}\,\ud v\ud x\ud t  =0.
\]
Since the boundary term disappeared, there is now no difficulty in truncating a function $\psi$ with unbounded first derivatives and passing to the limit.
\end{proof}

\begin{corollary} \label{cor:weak-ineq}
Let $f(t,x,v)$ be a weak solution of \eqref{eqn:boltzmann} satisfying the bounce-back, specular-reflection or in-flow boundary condition. Let $a(t)$ be any nonnegative Lipschitz function. In the case of the in-flow boundary, let us assume that the boundary value of $f$ on $\BC_-$ is smaller than $a(t)$ everywhere. Then, for any $\varphi \in C^1_c$ so that $\varphi \geq 0$, we have the inequality
\[ \begin{aligned}
&\iiint_{(0,T)\times \Omega\times\R^d} \left\{\frac 12 \left(f-a(t)\right)_+^2 \,\trp \varphi + \left(f-a(t)\right)_+ \Qff \varphi- a'(t) \left(f-a(t)\right)_+ \varphi \right\}\,\ud v\ud x\ud t \geq 0.
\end{aligned} \]
\end{corollary}

\begin{proof}
Depending on the boundary condition, we apply Lemma \ref{lem:chain-sub--sol-inflow}, Lemma \ref{lem:chain-weak-sol-bounce-back} or Lemma \ref{lem:chain-weak-sol-specular} with $\psi(t,f) = \frac 12 \left(f-a(t)\right)_+^2$.
\end{proof}

\subsection{On the Truncated Dissipation of Energy}

Let $f(t,x,v)$ be a function defined on $D$ that satisfies the bounce-back, specular-reflection, or in-flow boundary condition.

Let $a=a(t)\geq0$. We write $f=\fb+\fr$, where $\fr:=(f-a)_+$ and $\fb\leq a$. Clearly, if $f>a$, then $\fb=a$ and $\fr=f-a$; if $f\leq a$, then $\fb=f$ and $\fr=0$. Further, we define
\begin{align} \label{e:m}
    m(t): =\iint_{\Omega\times\R^d}\abs{\fr(t,x,v)}^2\ud v\ud x.
\end{align}

Since the function $f$ belongs to $L^2((0,T)\times\Omega \times \R^d)$, then $m \in L^1((0,T))$. 

\begin{lemma}\label{lem:energy} 
    Let $f(t,x,v)$ be a weak solution of \eqref{eqn:boltzmann} satisfying the bounce-back, specular-reflection, or in-flow boundary condition. In the case of in-flow boundary, we assume that the boundary value is smaller than $a(t)$ everywhere. Then for almost all $t_1,t_2 \in (0,T)$ with $t_1<t_2$ we have
    \[m(t_2)-m(t_1) \leq 2\iiint_{(t_1,t_2)\times\Omega\times\R^d} \left\{\fr\, \Qff - a'(t) \fr \right\} \, \ud v \ud x \ud t.\]
    holds for almost all $t\in(0,T)$.
\end{lemma}

\begin{proof}
Let $t_1$ and $t_2$ be two Lebesgue points of the $L^1$ function $m(t)$. For any $\ep>0$, we apply Corollary~\ref{cor:weak-ineq} for 
\[ \varphi_\ep(t) = \begin{cases}
     \frac 1 \ep (t_2-t) & \text{ if } t \in [t_2-\ep,t_2) \\
     \frac 1 \ep (t-t_1) & \text{ if } t \in (t_1,t_1+\ep] \\
     1 & \text{ if } t \in (t_1+\ep,t_2-\ep) \\
     0 & \text{ elsewhere.}
\end{cases}
\]
We obtain
\[
\iiint_{(0,T)\times \Omega\times\R^d} \frac 12 f_r^2 \left( \frac 1 \ep \one_{(t_1,t_1+\ep)} - \frac 1 \ep \one_{(t_2-\ep,t_2)} \right) + f_r \Qff \varphi_\ep - a'(t) f_r \varphi_\ep \ \ud v\ud x\ud t \geq 0.
\]
This is the same as
\[ \frac 1 \ep \int_{t_1}^{t_1+\ep} \frac{m(t)}2 \ud t - \frac 1\ep \int_{t_2-\ep}^{t_2} \frac{m(t)}2 \ud t + \iiint_{(0,T)\times \Omega\times\R^d} f_r \Qff \varphi_\ep - a'(t) f_r \varphi_\ep \,\ud v\ud x\ud t \geq 0.  \]
We conclude the proof taking the limit as $\ep \to 0$.
\end{proof}

Note that the right hand side of Lemma \ref{lem:energy} converges to zero as $t_2-t_1 \to 0$. Thus, the function $m(t)$ must be semicontinuous and its values are well determined for all $t \in (0,T)$.

\begin{corollary} \label{cor:cadlag}
The function $m(t)$ is almost everywhere equal to a c\`adl\`ag function with countably many jump discontinuities that are all negative.
\end{corollary}

\begin{proof}
Clearly, the right hand side in Lemma \ref{lem:energy} converges to zero as $t_2 \to t_1$ or $t_1 \to t_2$.
\end{proof}

After Corollary \ref{cor:cadlag}, it makes sense to refer to the values of $m(t)$ pointwise, and they satisfy the inequality of Lemma \ref{lem:energy} for all values of $t_1$ and $t_2$.

\section{Proof of the Main Theorem}

The purpose of this last section is to prove Theorem \ref{thm:main}. The strategy is to use Lemma \ref{lem:energy} to prove that $m(t)$ is nonincreasing as a function of $t$. Then we will see that $m(t)\to 0$ as $t \to 0$, concluding that $m \equiv 0$ and therefore $f \leq a(t)$. We need to first analyze the integrand in Lemma \ref{lem:energy} to prove that it is not positive.

\begin{lemma} \label{l:Q(f,f)fr}
Let $f :\R^d \to [0,\infty)$ be a nonnegative function in $L^1_2(\R^d) \cap LlogL(\R^d)$ satisfying the hydrodynamic bounds \eqref{assumption:hydrodynamic}. For any $a > 0$ large enough, we write 
$f_b(v) := \min(f(v),a)$ and $f_r(v) := \left(f(v)-a\right)_+$. Then for some constants $c>0$ and $C$ large, depending on the hydrodynamic bounds, we have
\[ \int_{\R^d} \Qff f_r \,\ud v \leq - c\, \|f_r\|_{L^p_n}^2 - c\, a^{1+2s/d} \|f_r\|_{L^1_\ell} + C \int_{\R^d} \left(f \ast |\cdot|^\gamma\right) f f_r \,\ud v. \]
Here, $p$ and $n$ are the exponents from Lemma \ref{lem:coercivity} and $\ell = \gamma+2s+2s/d$.
\end{lemma}

\begin{proof}
We split the left hand side using $\Qff = Q(f,f_b) + Q(f,f_r)$ and estimate a bound for each of the two terms.

Using \eqref{Q-fgh}, we write
\begin{align} \label{I-1}
    \int_{\R^d}Q(f,\fb)(v)\fr(v)\,\ud v
    =&\iint_{\R^d\times\R^d}\fr(v)\left(\fb(v')-\fb(v)\right)K_f(v,v')\,\ud v'\ud v\\
    &+c\int_{\R^d}\left(f\ast_v\abs{v}^{\gamma}\right)\fb(v)\fr(v)\,\ud v.\no
\end{align}
We estimate the first term in \eqref{I-1} using the properties of the non-degeneracy cone $\Xi(v)$ for $K_f$.

Recall that $K_f(v,v')\geq0$ everywhere and that by \eqref{K-lower-bound}
\begin{align*} 
    K_f(v,v') \gtrsim \langle v \rangle^{\gamma+2s+1} |v-v'|^{-d-2s}
    \;\;\text{ if }\; v'-v \in \Xi(v) .
\end{align*}
Moreover, observe that in the support of $\fr(v)$, we have $\fb(v)=a$. 
Also, $\fb(v')\leq a$ and $\fb\leq f$ everywhere, and thus $\nm{\fb}_{L^1_2(\R^d)}\ls (M_0+E_0)$.
For any fixed $v\in\R^d$, by taking $R=R(v)>0$ such that $R^d=Ca^{-1}\br{v}^{-1}(M_0+E_0)$ for some sufficiently large constant $C$, we ensure that $\fb(v')\leq \frac{a}{2}$ for at least half of the points $v'\in v + \left(B_R \cap \Xi(v)\right)$ and the measure $\abs{B_R \cap \Xi(v)} \approx R^d \langle v \rangle^{-1}$.

We therefore obtain
\begin{align*} 
    \int_{\R^d} \left(\fb(v')-\fb(v)\right) K_f(v,v')\,\ud v'
    &\leq\int_{v + (B_R \cap \Xi(v))}\left(\fb(v')-\fb(v)\right)K_f(v,v')\,\ud v'\\
    &\ls \langle v \rangle^{\gamma+2s+1} R^{-d-2s} \int_{v + (B_R \cap \Xi(v))}\left(\fb(v')-\fb(v)\right)\,\ud v' \\
    &\ls -a \langle v \rangle^{\gamma+2s} R^{-2s}
    \approx -a^{1+\frac{2s}{d}}\br{v}^{\gamma+2s+\frac{2s}{d}} = -a^{1+\frac{2s}{d}}\br{v}^\ell.
\end{align*}

Thus,
\begin{align} \label{tt 02}
\int_{\R^d}Q(f,\fb)(v)\fr(v)\,\ud v &\leq -a^{1+\frac{2s}{d}} \|f_r\|_{L^1_\ell}+c\int_{\R^d}\left(f\ast_v\abs{v}^{\gamma}\right)\fb(v)\fr(v)\,\ud v.
\end{align}

We then move on to estimate the term that involves $Q(f,f_r)$.
Following \eqref{Q-fgg}, we write
\begin{align*}
    \int_{\R^d}Q(f,\fr)(v)\fr(v)\,\ud v=&-\frac{1}{2}\iint_{\R^d\times\R^d}\abs{\fr(v')-\fr(v)}^2K_f(v,v')\,\ud v'\ud v\\
    &+c\int_{\R^d}\left(f\ast_v\abs{v}^{\gamma}\right)\abs{\fr(v)}^2\ud v.
\end{align*}
Applying Lemma~\ref{lem:coercivity} to $g=\fr$,
the first term above can be bounded as
\begin{align} \label{ttt}
    \iint_{\R^d\times\R^d}\abs{\fr(v')-\fr(v)}^2K_f(v,v')\,\ud v'\ud v &\gs 
    \nm{\fr}_{L_n^p}^{2-p}
    \cdot\! \int_{\left\{\abs{\fr(v)}\geq C_1\nm{\fr}_{L_n^p}\br{v}^k\right\}} \langle v \rangle^{np} |\fr(v)|^p \,\ud v \\
    &= \pnnm{\fr}{p}{n}^2 - \nm{\fr}_{L_n^p}^{2-p}
    \cdot\!\int_{\left\{\abs{\fr(v)}< C_1\nm{\fr}_{L_n^p}\br{v}^k\right\}} \langle v \rangle^{np} |\fr(v)|^p \,\ud v , \no
\end{align}
with $p$ and $n$ as in Lemma \ref{lem:coercivity} and $k=\frac{1}{2}\left(-\gamma-d+1\right)$. We also have 
\begin{align*}
    \int_{\left\{\abs{\fr(v)}< C_1\nm{\fr}_{L_n^p}\br{v}^k\right\}} \langle v \rangle^{np} |\fr(v)|^p \,\ud v 
    &\ls \nm{\fr}_{L_n^p}^{p-1} \cdot\!\int_{\R^d} \langle v \rangle^{np+(p-1)k} |\fr(v)| \,\ud v \\
    &\leq \nm{\fr}_{L_n^p}^{p-1} \nm{\fr}_{L^1_{\ell/2}} ,
\end{align*}
observing that the exponent $np+(p-1)k=\frac{1}{2}(1+\gamma-d)\leq \ell/2$.
Hence, we find that
\begin{align}\label{tt 03}
    \int Q(f,f_r) f_r \,\ud v  \leq &\,-c \pnnm{\fr}{p}{n}^2 + C \pnnm{\fr}{p}{n}\nm{\fr}_{L^1_{\ell/2}} + C \int_{\R^d}\left(f\ast_v\abs{v}^{\gamma}\right)\fr^2\,\ud v.
\end{align}

Combining \eqref{tt 02} and \eqref{tt 03}, we have
\begin{align}\label{tt 04}
    \int \Qff f_r \,\ud v \leq & -c \pnnm{\fr}{p}{n}^2- c\, a^{1+\frac{2s}{d}} \pnnm{\fr}{1}{\ell} + C \pnnm{\fr}{p}{n}\nm{\fr}_{L^1_{\ell/2}} + C\int_{\R^d}\left(f\ast_v\abs{v}^{\gamma}\right)f\fr\,\ud v
\end{align}

We observe that $\pnnm{\fr}{p}{n}\nm{\fr}_{L^1_{\ell/2}} \leq c \pnnm{\fr}{p}{n}^2 + c^{-1} \nm{\fr}_{L^1_{\ell/2}}^2$. Moreover, from the Cauchy–Schwarz inequality $\nm{\fr}_{L^1_{\ell/2}} \leq \nm{\fr}_{L^1}^{1/2} \nm{\fr}_{L^1_{\ell}}^{1/2}$. Therefore $\pnnm{\fr}{p}{n}\nm{\fr}_{L^1_{\ell/2}} \leq c \pnnm{\fr}{p}{n}^2 + C M_0 \nm{\fr}_{L^1_{\ell}}$. Thus, the third term in \eqref{tt 04} can be absorbed by the first two provided that $a$ is large enough. We finally simplify our estimate to
\[ \int \Qff f_r \,\ud v \leq -c \pnnm{\fr}{p}{n}^2- c a^{1+\frac{2s}{d}} \pnnm{\fr}{1}{\ell} + C\int_{\R^d}\left(f\ast_v\abs{v}^{\gamma}\right)f\fr\,\ud v. \]
\end{proof}

The right hand side in the inequality provided by Lemma \ref{l:Q(f,f)fr} contains two negative terms and a positive one. To prove our main theorem, we must estimate the postive term in terms of the negative terms. The next few lemmas provide some upper bounds that will be useful to that effect.

For the next lemmas, recall that $f = f_b + f_r$, where $f_b(v) = \min(f(v),a)$ and $f_r(v) = \left(f(v)-a\right)_+$. Moreover,
\begin{equation} \label{e:hydro-upper-bounds} \int_{\R^d} f(v) \,\ud v \leq M_0, \qquad \int_{\R^d} |v|^2 f(v) \,\ud v \leq E_0. \end{equation}

The following two lemmas are relatively standard. They are already proved in \cite{silvestre2016}. We include them here for completeness.

\begin{lemma} \label{l:conv-f-hard}
Let $f: \R^d \to [0,\infty)$ satisfy \eqref{e:hydro-upper-bounds}. Assume $\gamma \in [0,2]$. Then, for any $v \in \R^d$ we have
\[ \int_{\R^d} f(v-w) |w|^\gamma \,\ud w  \ls E_0 + \langle v \rangle^{2} M_0. \]
\end{lemma}

\begin{proof} The lemma follows by the following computation:
\begin{align*}
    \int_{\R^d}f(w)\abs{v-w}^{\gamma}\,\ud w\ls \int_{\R^d}f(w)\left(\abs{v}^{\gamma}+\abs{w}^{\gamma}\right)\,\ud w\ls M_0\br{v}^{\gamma}+E_0
    \ls \br{v}^{\gamma}.
\end{align*}
\end{proof}

\begin{lemma} \label{l:conv-fb}
Let $f: \R^d \to [0,\infty)$ satisfy \eqref{e:hydro-upper-bounds}. Assume $\gamma \in (-d,0]$. Then, for any $v \in \R^d$ we have
\[ \int_{\R^d} f_b(v-w) |w|^\gamma \,\ud w  \leq C M_0^{1+\gamma/d} a^{-\gamma/d}.
\]
\end{lemma}

\begin{proof}
For any $R>0$, we split the integral between $w \in B_R$ and the rest.
\begin{align*} \int_{\R^d} f_b(v-w) |w|^\gamma \,\ud w  &= \int_{B_R} f_b(v-w) |w|^\gamma \,\ud w + \int_{\R^d \setminus B_R} f_b(v-w) |w|^\gamma \,\ud w \\
&\leq a \int_{B_R} |w|^\gamma \ud w + R^\gamma \|f_b\|_{L^1} \\
&\ls a R^{d+\gamma} + R^\gamma M_0
\end{align*}
We finish the proof choosing $R = (M_0/a)^{1/d}$.
\end{proof}

\begin{lemma} \label{l:conv-fr}
Let $f: \R^d \to [0,\infty)$ satisfy \eqref{e:hydro-upper-bounds}. Let $\gamma<0$ and let $p$ and $n$ be the exponents of Lemma \ref{lem:coercivity}. Then, for any $v \in \R^d$, we have
\[ \int_{\R^d} f_r(v-w) |w|^\gamma \,\ud w  \leq C \left( \|f_r\|_{L^p_n}^{-2\gamma/(d+2s)} \|f_r\|_{L^1}^{1+2\gamma/(d+2s)} \langle v \rangle^m + \|f_r\|_{L^1} \right).
\]
Here, $C$ is a constant depending on the dimension $d$, $s$ and $\gamma$ only, and $m = 2\gamma n / (d+2s)$.

Note that $m \leq 0$ if $n\geq 0$. Moreover, we always have $m \leq 2s\ell/(d+2s)$.
\end{lemma}

\begin{proof}
For any $R \in (0,1)$, we split the integral between $w \in B_R$ and the rest. 
\begin{align*} \int_{\R^d} f_r(v-w) |w|^\gamma \,\ud w  &= \int_{B_R} f_r(v-w) |w|^\gamma \,\ud w + \int_{\R^d \setminus B_R} f_r(v-w) |w|^\gamma \,\ud w \\
&\leq \|f_r\|_{L^p(B_R(v))} \| |w|^\gamma \|_{L^{p'}(B_R)} + R^{\gamma} \|f_r\|_{L^1} \\
&\ls \langle v \rangle^{-n} \|f_r\|_{L^p_n} R^{\gamma+d/2+s} + R^{\gamma} \|f_r\|_{L^1}
\end{align*}

If $\|f_r\|_{L^1} \leq \langle v \rangle^{-n} \|f_r\|_{L^p_n}$, we take $R^{d/2+s} = \langle v \rangle^{n} \|f_r\|_{L^1} / \|f_r\|_{L^p_n}$. Otherwise, we take $R = 1$.

In the first case, we get
\[ \int f_r(v-w) |w|^\gamma \,\ud w \ls
\|f_r\|_{L^p_n}^{-2\gamma/(d+2s)} \|f_r\|_{L^1}^{1+2\gamma/(d+2s)} \langle v \rangle^{2 \gamma n/(d+2s)}
\]

In the second case, we get
\[ \int f_r(v-w) |w|^\gamma \,\ud w  \ls \|f_r\|_{L^1}. \]
In either case, we finish the proof. The inequalities indicated for $m$ are immediately verified from its formula provided that $\gamma+2s \geq 0$, after noticing that $-n < \ell/2$.
\end{proof}

\begin{lemma} \label{l:fr2-a}
Let $f_r:\R^d \to [0,\infty)$. Let $\ell = \gamma+2s +2s/d$ and $n$, $p$ be the exponents of Lemma \ref{lem:coercivity}. Then, for any $q \in \R$, we have
\[ \int_{\R^d} \langle v \rangle^{q} f_r^2 \,\ud v \lesssim \|f_r\|_{L^p_n}^{2d/(d+2s)} \|f_r\|_{L^1_m}^{4s/(d+2s)},\]
where $m = \frac 12 \left( \frac{d}{2s} (q - \gamma) - d + 1 + q \right)$.

In particular, for $q=\ell$, we get $m = (\ell+2)/2$, for $q=2s\ell/(d+2s)$, we get $m \leq (\ell+1)/2$, and for $q=0$, we get $m = \frac 12 \left( - \frac{d (\gamma+2s)}{2s} + 1 \right) \leq 1/2$.
\end{lemma}

\begin{proof}
Applying H\"older's inequality, we observe that
\begin{align*}
\int_{\R^d} \langle v \rangle^q f_r^2 \,\ud v &\leq \left( \int \langle v \rangle^{np} f_r^p \right)^{\alpha_1/p} \left( \int \langle v \rangle^{m} f_r \right)^{\alpha_2},
\end{align*}
provided that $\alpha_1 \geq 0$ and $\alpha_2 \geq 0$ satisfy
\begin{align*}
\frac{\alpha_1}p + \alpha_2 &= 1, \\
\alpha_1 + \alpha_2 &= 2, \\
\alpha_1 n + \alpha_2 m &= q.
\end{align*}
Given the choices $1/p = 1/2 - s/d$ and $n = (\gamma+2s-2s/d)/2$, there is only one choice of $\alpha_1$, $\alpha_2$ and $m$ that makes the three identities above hold. They are
\begin{align*} 
\alpha_1 &= \frac{2d}{d+2s}, \\
\alpha_2 &= \frac{4s}{d+2s}, \\
m &= \frac 12 \left( (q-\gamma-2s) \frac{d+2s}{2s} + \gamma+2s+1 \right).
\end{align*}
\end{proof}

\begin{lemma} \label{l:fr2-b}
Let $f_r:\R^d \to [0,\infty)$. Let $\ell = \gamma+2s +2s/d$ and $n$, $p$ be the exponents of Lemma \ref{lem:coercivity}. Then, for any any arbitrary small $\ep > 0$, there is a (presumably large) constant $C(\ep)$ so that
\[ \int_{\R^d} \langle v \rangle^{\ell} f_r^2 \,\ud v \lesssim \ep \|f_r\|_{L^p_n}^2 + C(\ep) \|f_r\|_{L^1_\ell} \|f_r\|_{L^1_2}.\]
\end{lemma}

\begin{proof}
Following Lemma \ref{l:fr2-a}, we have
\[ \int_{\R^d} \langle v \rangle^{\ell} f_r^2 \,\ud v \lesssim \|f_r\|_{L^p_n}^{2d/(d+2s)} \|f_r\|_{L^1_m}^{4s/(d+2s)},\]
where $m  = \ell/2+1$.

Consequently, for any $\ep>0$, there exists a constant $C$ so that
\[ \int_{\R^d} \langle v \rangle^{\ell} f_r^2 \,\ud v \leq \ep \|f_r\|_{L^p_n}^2 + C \ep^{-d/(2s)} \|f_r\|_{L^1_m}^2. \]

We use Cauchy–Schwarz inequality to estimate $\|f_r\|_{L^1_m}$ using $\|f_r\|_{L^1_\ell}$
\begin{align*}
\|f_r\|_{L^1_m} &= \int_{\R^d} \langle v \rangle^m f_r(v) \,\ud v \\
&\leq \left( \int_{\R^d} \langle v \rangle^\ell f_r(v) \,\ud v \right)^{1/2} \left( \int_{\R^d} \langle v \rangle^{2m-\ell} f_r(v) \,\ud v \right)^{1/2}.
\end{align*}
Recalling the formula for $m$ above, we observe that $2m-\ell = 2$. Therefore, we conclude
\[ \int_{\R^d} \langle v \rangle^{\ell} f_r^2 \,\ud v \leq \ep \|f_r\|_{L^p_n}^2 + C \ep^{-d/(2s)} \|f_r\|_{L^1_\ell} \|f_r\|_{L^1_2}. \]
\end{proof}

\begin{lemma} \label{l:badterm} Given any $\ep>0$, there exists a (presumably large) constant $C(\ep)$ (depending also on $M_0$ and $E_0$) so that
\begin{equation} \label{e:badterm}
\int_{\R^d} \left(f \ast |\cdot|^\gamma\right) f f_r \,\ud v \leq \ep \|f_r\|_{L^p_n}^2 + C(\ep) a^{1+\gamma_-/d} \|f_r\|_{L^1_\ell} .
\end{equation}
Here, $\gamma_- = -\gamma$ when $\gamma<0$ and $\gamma_- = 0$ when $\gamma \geq 0$.
\end{lemma}

\begin{proof}
We divide the proof into two cases depending on whether $\gamma \geq 0$ or $\gamma < 0$.

If $\gamma \geq 0$, we apply Lemma \ref{l:conv-f-hard} to get $f \ast |\cdot|^\gamma \lesssim \langle v \rangle^\gamma$. Thus,
\begin{align*}
\int_{\R^d} \left(f \ast |\cdot|^\gamma\right) f f_r \,\ud v &\ls \int_{\R^d} \langle v \rangle^\gamma \left(a+f_r\right) f_r \,\ud v \\
&\ls a \|f_r\|_{L^1_\gamma} + \int_{\R^d} \langle v \rangle^\gamma f_r^2 \,\ud v \\
\intertext{We use that $\gamma<\ell$ and Lemma \ref{l:fr2-b} together with $\|f_r\|_{L^1_2} \leq M_0 + E_0$.}
&\leq \ep \|f\|_{L^p_n}^2 + \left(a+C(\ep)\right) \|f_r\|_{L^1_\ell}
\end{align*}
and we finish the proof in the case $\gamma \geq 0$.

When $\gamma<0$, we write $f \ast |\cdot|^\gamma = f_b \ast |\cdot|^\gamma + f_r \ast |\cdot|^\gamma$. We estimate the first term using Lemma \ref{l:conv-fb} and the second term using Lemma \ref{l:conv-fr}.
\begin{align*}
\int_{\R^d} \left(f \ast |\cdot|^\gamma\right) f f_r \,\ud v &\ls a^{-\gamma/d} \int_{\R^d} \left(a+f_r\right) f_r \,\ud v \\
&\phantom{\ls} + \|f_r\|_{L^p_n}^{-2\gamma/(d+2s)} \|f_r\|_{L^1}^{1+2\gamma/(d+2s)} \left( \int_{\R^d} \langle v \rangle^{m} \left(a+f_r\right) f_r \,\ud v \right) \\
&\phantom{\ls} + \|f_r\|_{L^1_{\ell}} \left( \int_{\R^d} \left(a+f_r\right) f_r \,\ud v \right)  \\
\intertext{We use that $m \leq 2s\ell/(d+2s)$, and Lemma \ref{l:fr2-a} with $q=2s\ell/(d+2s)$ and $q=0$.}
&\ls a^{1-\gamma/d} \|f_r\|_{L^1} + a^{-\gamma/d}  \|f_r\|_{L^p_n}^{2d/(d+2s)} \|f_r\|_{L^1_{1/2}}^{4s/(d+2s)} \\
&\phantom{\ls} + a \|f_r\|_{L^p_n}^{-2\gamma/(d+2s)} \|f_r\|_{L^1}^{1+2\gamma/(d+2s)} \|f_r\|_{L^1_m} + \|f_r\|_{L^p_n}^{2(d-\gamma)/(d+2s)} \|f_r\|_{L^1}^{1+2\gamma/(d+2s)} \|f_r\|_{L^1_{(\ell+1)/2}}^{4s/(d+2s)} \\
&\phantom{\ls} + a \|f_r\|_{L^1}^2 + \|f_r\|_{L^1} \|f_r\|_{L^p_n}^{2d/(d+2s)} \|f_r\|_{L^1_{1/2}}^{4s/(d+2s)} \\
&=: {\rm (i)} + {\rm (ii)} + {\rm (iii)} + {\rm (iv)}  + {\rm (v)} + {\rm (vi)}.
\end{align*}

We must analyze each one of the six terms. The first one is clearly bounded by the second term in \eqref{e:badterm}. For (ii), we observe that both exponents are positive and add up to two, therefore
\begin{align*}
{\rm (ii)} = a^{-\gamma/d}  \|f_r\|_{L^p_n}^{2d/(d+2s)} \|f_r\|_{L^1_{1/2}}^{4s/(d+2s)} \leq \ep \|f_r\|_{L^p_n}^2 + C(\ep) a^{\frac{-\gamma(d+2s)}{2sd}} \|f_r\|_{L^1_{1/2}}^2
\intertext{Observe that $1/2 < 1+\ell/2$, and $\frac{-\gamma(d+2s)}{2sd} = - \gamma/(2s) - \gamma/d < 1-\gamma/d$ using that $\gamma+2s > 0$.}
\leq \ep \|f_r\|_{L^p_n}^2 + C(\ep) a^{1-\gamma/d} \|f_r\|_{L^1_{1+\ell/2}}^2 \leq \ep \|f_r\|_{L^p_n}^2 + C(\ep) a^{1-\gamma/d} \|f_r\|_{L^1_2} \|f_r\|_{L^1_\ell}.
\end{align*}
In the last inequality we used $\|f_r\|_{L^1_{1+\ell/2}}^2 \leq \|f_r\|_{L^1_2} \|f_r\|_{L^1_\ell}$, which follows by Cauchy–Schwarz.

The analysis of (iii) is very similar to (ii). We use that $m < 1+\ell/2$ and that the exponents are positive numbers that add up to two. Therefore,
\begin{align*}
{\rm (iii)} &\leq \ep \|f_r\|_{L^p_n}^2 + C(\ep) a^{\frac{d+2s}{d+\gamma+2s}} \|f_r\|_{L^1_{1+\ell/2}}^2 \\
\intertext{Observe that $(d+2s)/(d+2s+\gamma) = 1 - \gamma/(d+2s+\gamma) < 1-\gamma/d$ because $\gamma+2s > 0$.}
&\leq \ep \|f_r\|_{L^p_n}^2 + C(\ep) a^{1-\gamma/d} \|f_r\|_{L^1_2} \|f_r\|_{L^1_\ell}.
\end{align*}
We used Cauchy–Schwarz for the last inequality just like in the analysis of (ii).

To analyze (iv), note that $0 < (1+\ell)/2 \leq 2$. Therefore
\[ {\rm (iv)} \leq \|f_r\|_{L^p_n}^{2(d-\gamma)/(d+2s)} \|f_r\|_{L^1_{(1+\ell)/2}}^{1 + 2(\gamma+2s)/(d+2s)} \leq \|f_r\|_{L^p_n}^{2(d-\gamma)/(d+2s)} \|f_r\|_{L^1_{(1+\ell)/2}}^{2(\gamma+2s)/(d+2s)} \|f_r\|_{L^1_2}. \]
We use that $\|f_r\|_{L^1_2} \leq M_0 + E_0$, the remaining exponents add up to two, and $(1+\ell)/2 < 1 + \ell/2$ to get
\begin{align*}
{\rm (iv)} &\leq \ep \|f_r\|_{L^p_n}^2 + C(\ep) \|f_r\|_{L^1_{1+\ell/2}}^2 \leq \ep \|f_r\|_{L^p_n}^2 + C(\ep) \|f_r\|_{L^1_2} \|f_r\|_{L^1_\ell}.
\end{align*}
For (v), we observe that ${\rm (v)} \leq a M_0 \|f_r\|_{L^1_\ell}$. Finally, the analysis for (vi) is very similar (but simpler) to that of (iv).

Recalling that $\|f\|_{L^1_2} \leq M_0 + E_0$, we conclude that every term is bounded by the right hand side of \eqref{e:badterm} provided that $C(\ep)$ and $a$ are sufficiently large depending on $d$, $\gamma$, $s$, $M_0$ and $E_0$
\end{proof}

\begin{remark}
Reading the proof of Lemma \ref{l:badterm}, there seems to be a lot of room for the computation of the exponents of $\langle v \rangle$ in the weights of the inequalities. This is not surprising given that our $L^\infty$ bound in Theorem \ref{thm:main} is not meant to capture the sharp asymptotics as $|v| \to \infty$. Indeed, following \cite{imbert-mouhot-silvestre-decay2020}, we expect $f(t,x,v) \leq a(t) \langle v \rangle^{-q}$ for some exponent $q>0$.
\end{remark}

\begin{lemma} \label{l:mdecreasing}
Let $a(t) = C\left(1+t^{-\frac{d}{2s}}\right)$ as in Theorem \ref{thm:main}, and $m(t)$ be given by the formula \eqref{e:m}.
Then $m(t)$ is monotone decreasing with respect to $t$.
\end{lemma}

\begin{proof}
We point out that $m(t)$ is almost everywhere equal to a c\`adl\`ag function according to Corollary \ref{cor:cadlag}. This is the representative that we seek to prove that it is monotone decreasing.

From Lemma \ref{lem:energy}, we get that for almost every $t_1<t_2$,
\begin{align*}
m(t_2)-m(t_1) &\leq 2\iiint_{(t_1,t_2)\times\Omega\times\R^d} \left\{\fr\, \Qff - a'(t) \fr \right\}\, \ud v \ud x \ud t \\
\intertext{Using Lemmas \ref{l:Q(f,f)fr}, \ref{l:badterm}, and the fact that $-a' \leq C^{-2s/d} a^{1+2s/d}$, then for $C$ large enough, we also have $a(t) \geq C$ and }
&\ls   \iint_{(t_1,t_2)\times\Omega} \left\{ -c\,  \|f_r\|_{L^p_n}^2 - c\, a^{1+2s/d} \|f_r\|_{L^1_\ell} \right\}\, \ud x \ud t \leq 0.
\end{align*}
\end{proof}

\begin{proof}[Proof of Theorem~\ref{thm:main}]
We intend to prove $m(t) \equiv 0$, which implies $\nm{f(t)}_{L^\infty_{x,v}}\leq a(t)$ for all $t>0$.

Note that $f\in L^2_{t,x}N_v^{s,\gamma}$ implies $f(t)\in L^2_{x,v}$ for a.e. $t\in\R_+$. Hence, for any $\delta>0$, there exists $t_0\in(0,\delta)$ such that $f(t_0)\in L^2_{x,v}$. For any $t_1\in(0,t_0)$, denote
\[ m_{t_1}(t):=\iint_{\Omega\times\R^d}\left(f(t,x,v)-a(t-t_1)\right)_+^2\,\ud v\ud x.\]
$m_{t_1}(t)$ can be regarded as a shifted version of $m(t)$ which starts from $t_1$ instead of $0$. Lemma \ref{l:mdecreasing} applies to $m_{t_1}$ just as well, so we deduce that $m_{t_1}(t)$ is monotone decreasing for $t \in (t_1,\infty)$.

Since $\lim_{t\rt0}a(t)=\infty$, we have
\[\lim_{t_1\rt t_0}m_{t_1}(t_0)=0.\]
Hence, for any $\ep>0$, there exists $t_1\in(0,t_0)$ such that $m_{t_1}(t_0)<\ep$. 

Note that we may choose $0<t_1<t_0$ to be regular points of $f$ in the sense of Lebesgue differentiation as a function $f : (0,T) \to L^2(\Omega \times \R^d)$. In that way, the corresponding value of $m(t_1)$ and $m(t_2)$ is given literally by \eqref{e:m} without the need of a modification of either $f$ or $m$ in sets of measure zero.

Based on the monotonicity of $a(t)$ and $m_{t_1}(t)$, we have $m_{t_0}(t)\leq m_{t_1}(t)\leq m_{t_1}(t_0)<\ep$ for all $t\in(t_0,\infty)$.
Due to the arbitrariness of $\ep>0$, we have $m_{t_0}(t)=0$ for all $t\in(t_0,\infty)$.
Finally, due to the arbitrariness of $\delta>0$, we have $m(t)=0$ for all $t\in(0,\infty)$.

The fact that $m(t) \equiv 0$ implies $f(t,x,v) \leq a(t)$ almost everywhere, which is the result of Theorem~\ref{thm:main}.
\end{proof}

\bibliographystyle{plain}
\bibliography{ubbd}
\index{Bibliography@\emph{Bibliography}}%

\end{document}